\documentclass[11pt,reqno]{amsart}
\usepackage{graphicx,xcolor,mathtools,pifont}
\usepackage{amssymb,amsmath,amsthm,mathrsfs,paralist,bm,esint,setspace,accents}
\RequirePackage[colorlinks,%
	linkcolor=orange!80!black,%
	citecolor=cyan!80!black,%
	urlcolor=cyan!80!black]{hyperref}
\usepackage{cite}
\newcommand{\cev}[1]{\reflectbox{\ensuremath{\vec{\reflectbox{\ensuremath{#1}}}}}}
\usepackage[cal=dutchcal,
	scr=boondoxo,scrscaled=1.05,
	bb=boondox,bbscaled=1.05]{mathalfa}
\usepackage{cases}

\usepackage{color}
\usepackage{setspace}
                                   
\let\OLDthebibliography\thebibliography
\renewcommand\thebibliography[1]{
  \OLDthebibliography{#1}
  \setlength{\parskip}{3pt}
  \setlength{\itemsep}{0pt plus 0.3ex}
}

\DeclareMathOperator{\Var}{Var}
\DeclareMathOperator{\Cov}{Cov}

\DeclareMathOperator{\U0}{\mathbb{U}_0}
\DeclareMathOperator{\bS}{\mathbb{S}}

\newcommand{\<}{\langle}
\renewcommand{\>}{\rangle}

\newcommand{\N}{\mathbb{N}}
\newcommand{\Z}{\mathbb{Z}}
\newcommand{\Q}{\mathbb{Q}}
\newcommand{\R}{\mathbb{R}}

\newcommand{\lip}{\text{\rm Lip}}

\renewcommand{\P}{\mathrm{P}}
\newcommand{\E}{\mathrm{E}}

\newcommand{\cA}{\mathcal{A}}

\newcommand{\sI}{\mathscr{I}}

\newcommand{\sD}{\mathscr{D}}

\newcommand{\sC}{\mathscr{C}}

\newcommand{\sR}{\mathscr{R}}

\newcommand{\sF}{\mathscr{F}}

\newcommand{\sS}{\mathscr{S}}

\newcommand{\1}{\mathbb{1}}
\renewcommand{\d}{{\rm d}}

\newcommand{\e}{{\rm e}}
\renewcommand{\geq}{\geqslant}
\renewcommand{\leq}{\leqslant}
\renewcommand{\ge}{\geqslant}
\renewcommand{\le}{\leqslant}
\title[Ergodic theory of SPDEs]{The ergodic theory of SPDEs\\in a weak-noise regime} 
\thanks{Research supported in part by the United States National Science Foundation (DMS-2245242) and the Indian ANRF grant CRG/2023/002667}
\author[M. Joseph]{Mathew Joseph}
\address{Statmath Unit, Indian Statistical Institute,
	8th Mile Mysore Road, RVCE Post, Bengaluru 560059, India}
\email{m.joseph@isibang.ac.in}
\author[D. Khoshnevisan]{Davar Khoshnevisan}
\address{Department of Mathematics, University of Utah,
	Salt Lake City, UT 84112-0090, USA}
\email{davar@math.utah.edu}

\author[K. Kim]{Kunwoo Kim}
\address{Department of  Mathematics, Pohang University of Science and Technology (POSTECH), 
	Pohang, Gyeongbuk, 37673, Korea}
\email{kunwoo@postech.ac.kr}
\author[C. Mueller]{Carl Mueller}
\address{Department of Mathematics, University of Rochester,
	Rochester, NY 14627, USA}
\email{carl.e.mueller@rochester.edu}
\date{March 18, 2026}
\keywords{Stochastic partial differential equations, 
    ergodicity, invariant measures, weak noise}
\subjclass[1991]{ Primary, 60H15; Secondary, 37C40, 37L40}
\let\oldtocsection=\tocsection
\let\oldtocsubsection=\tocsubsection

\renewcommand{\tocsection}[2]{\hspace{0em}\oldtocsection{#1}{#2}}
\renewcommand{\tocsubsection}[2]{\hspace{2.5em}\oldtocsubsection{#1}{#2}}

\newtheorem{stat}{Statement}[section]
\newtheorem{proposition}[stat]{Proposition}

\newtheorem*{ithm}{Informal Theorem}
\newtheorem{corollary}[stat]{Corollary}
\newtheorem{theorem}[stat]{Theorem}
\newtheorem{lemma}[stat]{Lemma}

\theoremstyle{definition} 
\newtheorem{definition}[stat]{Definition}
\newtheorem{assumption}[stat]{Assumption}

\newtheorem{remark}[stat]{Remark}

\newtheorem{conjecture}{Conjecture}

\numberwithin{equation}{section}

\begin{document}
\begin{abstract}
	Consider a parabolic SPDE 
	\[
		\partial_t u = \Delta u + \sigma(u)\eta,
	\]
	on
	$(0\,,\infty)\times\R^d$, where
	$\eta$ is a centered, generalized Gaussian noise with 
	$\Cov[\eta(t\,,x)\,,\eta(s\,,y)]=\delta_0(t-s)\Lambda(x-y)$ for a tempered Borel
	measure $\Lambda$ that is positive definite and satisfies a mild weak-noise.
	The existence of invariant measures of versions of these types of SPDEs
	has been studied at great length, particularly in the ``weak-noise regime''; see for example
	Assing and Manthey \cite{AssingManthey2003},
	Chen and Eisenberg \cite{ChenEisenberg2024},
	Chen, Ouyang, Tindel, and Xia \cite{ChenOuyangTindelXia2024},
	Eckmann and Hairer \cite{EckmannHairer2001},
	Misiats and Stanzhytskyi \cite{MSY2020},
	Yu Gu and Jiawei Li \cite{GuLi2020}, and
	Tessitore and Zabczyk \cite{TessitoreZabczyk1998}.
	Here, we characterize all annealed, ergodic, invariant measures for the above
	SPDE in the weak-noise regime.
\end{abstract}
\maketitle
\tableofcontents

\section{Introduction}
Consider the initial-value problem,
\begin{equation}\label{SPDE}
	\partial_t u = \Delta u + \sigma(u)\eta\qquad\text{on $(0\,,\infty)\times\R^d$},
\end{equation}
subject to  $u(0)=u_0$ for a suitable initial profile $u_0$.
We adhere to the general theory of Dalang \cite{Dalang1999} and assume that the forcing term
$\eta$ is a centered, generalized Gaussian random field whose covariance measure is
given by
\[
	\Cov[\eta(s\,,y)\,,\eta(t\,,x)] = \delta_0(t-s) \Lambda(x-y)
	\qquad\forall s,t\ge0,\
	x,y\in\R^d.
\]
Moreover, the spatial covariance of the noise
$\Lambda=\sF\mu$ denotes the Fourier transform of
a symmetric, tempered, positive-definite Borel measure $\mu$ 
on $\R^d$, where $\sF$ denotes the Fourier transform on $\R^d$, normalized so that 
\[
	(\sF f)(\xi) =\int_{\R^d} \e^{-ix\cdot \xi} f(x) \, \d x 
	\qquad\forall  f\in L^1(\R^d),\ \xi\in\R^d.
\]
We assume additionally that the nonlinearity $\sigma: \R \to \R$ in \eqref{SPDE}
is a Lipschitz continuous function,  that $u_0:\R^d\to\R$ is a nonrandom,
bounded, and measurable function, and  that
\begin{equation}\label{cond:Dalang}
	\int_{\R^d}\frac{\mu(\d\xi)}{1+\|\xi\|^2}<\infty.
\end{equation}
In this way, we may deduce from the theory of Dalang \cite{Dalang1999}  that \eqref{SPDE}
is \emph{well posed}. By this we mean that, up to a modification, \eqref{SPDE} 
has exactly, and only, one random-field solution 
$u$ that is mild [see \eqref{mild} below] and satisfies
\begin{equation}\label{L2:bd}
	\adjustlimits\sup_{t\in(0,T)}\sup_{x\in\R^d}\E\left( |u(t\,,x)|^2\right)<\infty
	\quad\forall T>0.
\end{equation}

From now on, we always assume condition \eqref{cond:Dalang}.
This ensures that \eqref{SPDE} is well posed and permits us
to now study its ergodic theory.
In order to do that, we change our point of view slightly
and consider \eqref{SPDE}
as a dynamical description of a suitable infinite-dimensional Markov process.
Equivalently, \eqref{SPDE} can be viewed as a system of stochastic PDEs,
indexed by a collection of initial data that can be random, and
independent of the noise $\eta$. 

Let $(\Omega\,,\cA,\P)$ denote the underlying probability space.
In order to avoid unpleasant measure-theoretic
obstructions, we consider only measurable initial profiles 
$u_0:\Omega\times\R^d\ni (\omega\,,x)\mapsto \R$
that are independent of $\eta$. A routine extension
of the theory in Ref.~\cite{Dalang1999} then shows that \eqref{SPDE}
is well posed -- in exactly the same sense as before
-- provided that, in addition to \eqref{cond:Dalang},
\begin{equation}\label{L2:bd:u_0}
	\sup_{y\in\R^d}\E\left(|u_0(y)|^2\right)<\infty.
\end{equation}
Dalang's construction of the stochastic integral ({\it ibid.}) also implies that
$u(t):\Omega\times\R^d\ni (\omega\,,x)\mapsto \R$
is measurable for every $t>0$. Moreover, thanks to \eqref{L2:bd} and \eqref{L2:bd:u_0},
we can start the dynamics \eqref{SPDE} -- using an independent copy of the noise
$\eta$ -- started from $u(t)$. A standard computation then shows 
that the post-$t$ process $(s\,,x)\mapsto u(t+s\,,x)$
satisfies \eqref{SPDE}, started from initial data $u(t)$,
and with $\eta$ replaced by a copy of $\eta$ that is independent
of $u(t)$. In other words, our point of view of allowing jointly measurable
random initial profiles $u_0$ renders the dynamics in \eqref{SPDE}
as Markovian even though that Markov process might not
take values in a nice space, such as the space of continuous functions,
under the type of minimal hypotheses that are considered here.

Nevertheless, we are interested in understanding the structure
of the invariant measure(s) of those
Markovian dynamics under minimal weak-noise hypotheses on the noise,
equivalently the measures $\Lambda$ and/or $\mu$.
Such matters have been pursued in great length in the technically
simpler case that the role of $(\R^d\,,\Delta)$
in \eqref{SPDE} is replaced $(\Z^d\,,\Delta_{\Z^d})$ where $\Delta_{\Z^d}$
denotes the graph Laplacian on $\Z^d$;
see for example, Carmona and Molchanov \cite{CM},
Cox and Greven \cite{CoxGreven1994},
Greven and den Hollander \cite{GrevenDenHollander2007}, and 
Shiga \cite{Shiga1992}. In many of these papers, the analysis is further
restricted to the case that $\sigma$ is linear; that is 
the \emph{parabolic Anderson model} driven by $\eta$; see  \cite{CM}. 

Altogether, the mentioned ``semidiscrete setting''
includes infinitely-many interacting
diffusions, and the end result of the mentioned theory is that
one can often identify explicit minimal weak-noise conditions under which
there is a continuum of time-ergodic  invariant measures. Moreover,
as was shown in Deuschel \cite{Deuschel1994}, convergence
to stationarity is typically slow
as there often is no spectral gap in the weak-noise regime. 
In physical terms, the slow rate of convergence to stationarity then
would imply that the invariant measures are ``annealed.'' This notion
can be made  precise; 
see the discussion around Definition \ref{def:anneal} below
for more details.

Tessitore and Zabczyk \cite{TessitoreZabczyk1998} initiated an
analogous ergodic theory of \eqref{SPDE} in the present, continuum, setting 
by showing that \eqref{SPDE}
generally has invariant measures under certain technical conditions
and when the noise is sufficiently weak. This was followed by
a series of results where the existence of an invariant measure
is proved in a variety of similar settings. See, for example,
Assing and Manthey \cite{AssingManthey2003}, 
Chen and Eisenberg \cite{ChenEisenberg2024},
Chen, Ouyang, Tindel, and Xia \cite{ChenOuyangTindelXia2024},
Eckmann and Hairer \cite{EckmannHairer2001},
Yu Gu and Li \cite{GuLi2020}, and
Misiats, Stanzhytskyi, and Yip \cite{MSY2020}.

In the present work, we continue and complement
the mentioned literature by characterizing, using more precise language
than we have  used in the above, all ``annealed,'' time-ergodic
invariant measures for \eqref{SPDE} under minimal hypotheses.
The reader might notice that we write ``time ergodic'' in place of the more common
``ergodic.'' This is because the invariant measures of \eqref{SPDE}
live on the sample space $\R^{(\R^d)}$ and, as such, can be (and often are)
also laws of spatially stationary and ergodic random fields that are indexed
by $\R^d$. Because the implied connections between time ergodicity and
spatial ergodicity turn out to be central to parts of our work,
we refer to the ergodicity -- for example of the Markov process $\{u(t)\}_{t\ge0}$ --
more precisely as time ergodicity in order
to distinguish it from the notion of spatial ergodicity.

As was mentioned earlier, we wish to work in the ``weak-noise regime.''
By this we mean precisely  the following:
\begin{equation}\label{COND:WN}
	\int_{\R^d}\frac{\mu(\d z)}{\|z\|^2}<\infty
	\quad\text{and}\quad
	\frac{\lip_\sigma^2}{2}\int_{\R^d}\frac{\mu(\d z)}{\|z\|^2}< 1,
\end{equation}
where $\lip_\sigma=\sup_{x\neq y}|\sigma(x)-\sigma(y)|/|x-y|$ denotes
the optimal Lipschitz constant of $\sigma$. Of course when $\lip_\sigma>0$, that is
when $\sigma$ is not a constant, \eqref{COND:WN} can be written more succinctly
as $\int_{\R^d}\|z\|^{-2}\,\mu(\d z)< 2/\lip_\sigma^2$.\footnote{
	One can check directly that \eqref{COND:WN} can only be valid
	in transient dimensions $d\ge3$. Indeed, \eqref{mathscr(E)} below and
	the Tonelli theorem together imply that
	$\int_{\R^d}\|z\|^{-2}\,\mu(\d z) = 4(\mathscr{v}*\Lambda)(0),$
	for the zero-potential density $\mathscr{v}(x)=\int_0^\infty p_s(x)\,\d s$,
	which is infinite everywhere when $d=1$ or $d=2$.
}

Condition \eqref{COND:WN} has been identified already
in Chen and Eisenberg \cite[eq.s~(1.10a), (1.10b)]{ChenEisenberg2024},
and it is exact for the validity of $L^2$ methods, as will be pointed out also in 
Lemma \ref{lem:bdd:L2} below.

It is possible to apply the identities in
\eqref{mathscr(E)} below in order to see that
Condition \eqref{COND:WN} is precisely
the continuous counterpart of the
weak noise condition of the literature on particle systems, as can be found 
for example in Cox and Greven \cite[Theorem 1]{CoxGreven1994},
Greven and den Hollander
\cite[eq.~(1.21)]{GrevenDenHollander2007}, and
Shiga \cite[eq.~(1.10)]{Shiga1992}.

\begin{assumption}\label{ass:WN}
	Unless it is stated explicitly to the contrary, 
	from now on we assume that
	\eqref{COND:WN} holds. It should be clear that
	\eqref{COND:WN} subsumes \eqref{cond:Dalang}.
\end{assumption}

Recall that a spatial random field $u_0=\{u_0(x)\}_{x\in\R^d}$ is
called:
\begin{itemize}[\noindent$\bullet$]
	\item Stationary, or \emph{spatially stationary}, if $\{u_0(x)\}_{x\in\R^d}$
   		and $\{u_0(x+y)\}_{x\in\R^d}$ have the same finite-dimensional
		distributions for every $y\in\R^d$.
	\item Weak stationary, or \emph{weakly spatially stationary}, if $u_0(x)\in L^2(\Omega)$
      	for every $x\in\R^d$, $\E[u_0(x)]$
		does not depend on $x\in\R^d$, and $\Cov[u_0(x)\,,u_0(y)]
		=\Cov[u_0(0)\,,u_0(y-x)]$ for every $x,y\in\R^d$.
\end{itemize}

The principal aim of this paper is to state and prove a 
rigorous version of the following
whose notation and content is motivated by the work of Shiga
\cite{Shiga1992} on discrete systems.

\begin{ithm}
	If the weak-noise condition \eqref{COND:WN} holds, then:
	\begin{enumerate}[\rm(a)]
		\item For every $\theta\in\R$, the law of the solution $u(t)$ of \eqref{SPDE}
			at time $t>0$, started at 
			$u_0\equiv\theta$, converges weakly to a probability measure $\nu_\theta$
			as $t\to\infty$.
		\item $\{\nu_\theta\}_{\theta\in\R}$ are mutually singular.
		\item For every $\theta\in\R$, $\nu_\theta$ is a time-ergodic, annealed
			invariant measure for \eqref{SPDE}, and is the law of a spatially
			stationary random field.
		\item If $\nu$ is an annealed
			invariant measure for \eqref{SPDE}, then 
			$\nu=\nu_\theta$
			for $\theta=\int h(0)\,\nu(\d h)$.
	\end{enumerate}
\end{ithm}

The preceding is labeled as an informal theorem since, among other things, its precise statement
requires the introduction of a suitable topology for a proper description of 
terms such as ``law,''  ``weak convergence,'' and so on.
It also requires a rigorous definition of ``annealed'' random fields. All of this will 
be done in a series of steps that begin in the next section and culminate in Theorem \ref{th:main},
which is the precise form of the above Informal Theorem. 

We will see in Proposition \ref{pr:anneal:LLN} below that if an invariant measure
$\nu$ is the law of a  weakly spatially stationary random 
field $\{u_0(x)\}_{x\in\R^d}$, then $\nu$ is annealed if and only if
the following law of large numbers holds:
\begin{equation}\label{LLN}
	R^{-d}\int_{[-R,R]^d} u_0(x)\,\d x \xrightarrow{R\to\infty}
	\E[ u_0(0)]\qquad\text{in $L^2(\Omega)$}.
\end{equation}
In particular, part (d) of the Informal Theorem immediately yields the following.

\begin{corollary}\label{cor:unique}
	Choose and fix an arbitrary $\theta\in\R$. Then, $\nu_\theta$
	is the only mean-$\theta$ time-ergodic invariant measure among all
	laws of spatially stationary, spatially ergodic random fields should there
	be any. 
\end{corollary}
A semi-discrete version of Corollary \ref{cor:unique}, valid for space-time white noise
on $\R_+\times\Z^d$, 
appears earlier in Shiga \cite[Theorem 1.1]{Shiga1992}. 
Unfortunately, such statements are conditional assertions only
since we do not know \emph{a priori}
that $\nu_\theta$ is spatially ergodic. 
Theorem \ref{th:ergodic} below states that a stronger weak-noise
condition than \eqref{COND:WN} indeed --
see \eqref{cond:ergodic} -- ensures that $\nu_\theta$ is spatial ergodicity.
In this connection, let us state two conjectures.

\begin{conjecture}\label{conj1}
	We believe that \eqref{cond:ergodic} is suboptimal for
	the spatial ergodicity of the $\nu_\theta$s, and yet 
	\eqref{COND:WN} is not sufficient
	for the same ergodic property to hold.
\end{conjecture}

We have no rigorous arguments toward establishing Conjecture \ref{conj1}. 

\begin{conjecture}\label{conj2}
	We believe that every annealed time-ergodic invariant measure is the law of
	a  weakly spatially stationary random field.
\end{conjecture} 

In Lemmas \ref{lem:WS:const} and \ref{lem:WS:PAM} 
below, we verify Conjecture \ref{conj2} 
respectively when $\sigma$ is a constant (the Edwards-Wilkinson model) and when
$\sigma$ is linear (the parabolic Anderson model). 
The general case eludes us. 

Choose and fix some $\theta>0$. Let us reiterate that the Informal Theorem and its rigorous version
(Theorem \ref{th:main}) say that $\nu_\theta$ is
the only mean-$\theta$, time-ergodic,
annealed invariant measure. Moreover, we have a simple algorithm for approximate sampling
from the invariant measure $\nu_\theta$: Simply start \eqref{SPDE} from constant initial
data $\theta$, and run the SPDE up to a long time $t\gg1$. The random field $u(t)$
behaves as an approximate sample from $\nu_\theta$.

As we shall see, this algorithm is exact when $\sigma$ is a constant
(see Section \ref{sec:Gaussian}),
and can be greatly simplified in the case that
$\sigma$ is linear (see Section \ref{sec:PAM}). There are a relatively small
number of examples of infinite-dimensional Markov 
processes with easy-to-access invariant measures,
mostly in a discrete setting, cast for interacting 
particle systems. The present paper
adds to the literature many new example of such Markov processes,
one for every SPDE \eqref{SPDE} 
that satisfies \eqref{COND:WN}. 
Furthermore, our results will show that
samples from these invariant measures are (pointwise-defined)
random fields.
This identification with random fields 
somewhat strengthens the bulk of the mentioned literature 
which is concerned with
producing invariant measures on Hilbert space, 
specifically, (possibly weighted) $L^p$-spaces
[$p\in\{2\,,\infty\}$].

Finally, we mention that our method
includes the introduction of a new,  robust duality argument which is
easy to understand. In fact, our duality closely resembles
finite-dimensional duality results such as those in Nagasawa
\cite{Nagasawa1964}. We believe this duality method might have other uses as well.

Let us conclude the introduction with a brief outline of the paper.
In Section \ref{sec:info} we discuss some of the SPDE background necessary to tackle the problem at hand. Among other things, we introduce measure-theoretic descriptions of a good state space $\bS_d$ of the infinite-dimensional Markov process $\{u(t)\}_{t\ge0}$, together with
a reasonably nice sigma-algebra $\sS_d$. This construction turns out to be non-trivial since the solution to \eqref{SPDE} does not have much regularity solely under condition \eqref{COND:WN}. 
In Section \ref{sec:duality} we introduce a duality argument 
that is critical to our forthcoming analysis. We also describe some 
ergodic-theoretic facts that are valid for the minimally regular 
Markov processes of this paper. 

Subsequently, we introduce annealed invariant measures in
Section \ref{sec:main}. The exact, slighly improved,
version of the previous Informal Theorem (Theorem \ref{th:main})
can be found in Section \ref{sec:main} as well.
Section \ref{sec:Linear} is dedicated to 
the proof of Conjecture \ref{conj2} in the special case that $\sigma$ is either a constant
or is linear. 
Theorem \ref{th:main} is proved later in Section \ref{sec:Proof}.

Section \ref{sec:Holder} answers a question that was posed to us by 
Samy Tindel. Namely, Section \ref{sec:Holder} presents a weak-noise condition 
-- more restrictive than \eqref{COND:WN} --
that ensures that the invariant measures of \eqref{SPDE} are supported in the space of continuous functions. 
In this type of setting, the technical work of the first few sections, often involving delicate
coupling arguments, can be simply replaced
by known ergodic-theoretic methods \cite{DZ}, and our work can be simplified greatly though
we leave those simplifications to the interested reader.
Afterward, we present a more succint description of $\nu_\theta$ in the special cases that $\sigma
\equiv\text{const}$ (Section \ref{sec:Gaussian}) and $\sigma(z)\propto z$ 
(Section \ref{sec:PAM}). We conclude this paper in
Section \ref{sec:ergodicity} with a discussion of the previously mentioned
spatial ergodicity of the invariant measures of \eqref{SPDE}.

\section{Background information}\label{sec:info}

Our characterization of the annealed invariant measures for
\eqref{SPDE} has a number of technical
prerequisites that we include as part of the discussion of this section.
Let us begin with the basic SPDE \eqref{SPDE}. 

As was mentioned in the Introduction, Dalang \cite{Dalang1999} has shown that,
under condition \eqref{cond:Dalang}, there exists a predictable space-time
random field $u=\{u(t\,,x)\}_{t>0,x\in\R^d}$ -- see \cite{Dalang1999,Walsh} --
that satisfies \eqref{SPDE}, started from a non-random $u_0\in L^\infty(\R^d)$, in 
mild (or integral, or ``Duhamel'') form. This means that $u$ solves
the random integral equation,
\begin{equation}\label{mild}
	u(t\,,x) = (p_t*u_0)(x) + \int_{(0,t)\times\R^d} p_{t-s}(y-x)\sigma(u(s\,,y))\,
	\eta(\d s\,\d y),
\end{equation}
where the stochastic integral with respect to $\eta$ is understood as an
It\^o-Walsh type stochastic integral, and
\begin{equation}\label{p}
	p_t(x) = (4\pi t)^{-d/2}\exp\left( -\frac{\|x\|^2}{4t}\right)\qquad\forall
	t>0,x\in\R^d.
\end{equation}
Dalang's theory ({\it ibid.})
also yields the \emph{a priori} bound \eqref{L2:bd} for $u$,
and assures of the existence of a unique (up to a modification) such solution
subject to the integrability condition
\eqref{L2:bd}. Moreover, Dalang has proved that \eqref{cond:Dalang}
is necessary as well as sufficient for well posedness when $\sigma$ is a constant.

As was alluded to in the Introduction, in order to study the
ergodic theoretic properties of the dynamics in
\eqref{SPDE}, we will need to be able to start \eqref{SPDE}
according to a random initial data $u_0$ which, for us, will always
be a pointwise-defined random field. The extension to random $u_0$
is usually done without explicit mention since $u_0$ will necessarily 
always be independent of the noise $\eta$ and so one can simply condition
on $u_0$. However, there is a delicate measure-theoretic matter that
we need to pinpoint in this particular setting. Thus, let us start with
an explicit definition of the type of random initial profiles that we plan to study.
Recall that the underlying probability space is denoted
throughout by $(\Omega\,,\cA,\P)$.

\begin{definition}\label{def:U0}
	Throughout, let $\U0$ denote the collection of all real-valued, spatial 
	random fields $u_0=\{u_0(x)\}_{x\in\R^d}$ such that
	$u_0:\Omega\times\R^d\to\R$ is (jointly) measurable,
	and satisfies $\|u_0\|_{\U0}<\infty$, where
	\[
		\|u_0\|_{\U0}^2 = \sup_{x\in\R^d}\E\left( |u_0(x)|^2\right).
	\]
\end{definition}

In the following, we will tacitly use the readily checkable facts that:
(a) $\U0$ is a linear space;  (b) $\|\,\cdots\|_{\U0}$ defines a norm
on $\U0$; and (c) $(\U0\,,\|\,\cdots\|_{\U0})$ is
a Banach space.
The following is an immediate extension of Dalang's theory \cite{Dalang1999}.

\begin{proposition}\label{pr:Dalang}
	Consider the SPDE \eqref{SPDE}, started at some $u_0\in\U0$,
	and with an independent noise $\eta$. Then:
	\begin{enumerate}[\noindent\rm (a)]
      	\item \eqref{SPDE} has a mild solution
            	$u$ such that $u(t)\in\U0$ for all $t>0$. 
            	In fact, $u$ satisfies \eqref{L2:bd}; equivalently put,
            	$\sup_{t\in(0,T)}\|u(t)\|_{\U0}<\infty$ for all $T>0$.
            \item The random field $u$ is, up to modification, the only mild solution
            	that satisfies \eqref{L2:bd}. 
      	\item For every fixed, deterministic $t>0$,
		the post-$t$ process $s\mapsto u(t+s)$
            	satisfies \eqref{SPDE}, started from initial profile $u(t)\in\U0$, and 
			with $\eta$ replaced by a copy $\eta_t$
            	of $\eta$ that is independent
            	of $u(t)$.
	\end{enumerate}
\end{proposition}
The proof requires making only small modifications to the arguments
of Dalang ({\it ibid.}) to adjust for the fact that $u_0$
can be random. As such, the proof is left to the interested reader.
To be sure, however, we add that the noise $\eta_t$, mentioned in the latter
portion of Proposition \ref{pr:Dalang}, is the noise obtained by time-shifting
$\eta$  by  $t$ units; that is, $\eta_t$ is defined canonically via
\[
	\int_{(0,\infty)\times\R^d} f(t+s)g(y)\,\eta_t(\d s\,\d y) =
	\int_{(t,\infty)\times\R^d} f(s)g(y)\,\eta(\d s\,\d y),
\]
for every non-random, continuous, and
compactly supported function $f:(0\,,\infty)\to\R$
and for every rapidly decreasing test function $g$ on $\R^d$.

We can now interpret the weak-noise condition \eqref{COND:WN}.
For this particular result, we temporarily suspend the  \emph{a priori} 
assumption that
\eqref{COND:WN} holds (see Assumption \ref{ass:WN}).

\begin{lemma}\label{lem:bdd:L2}
	Let $u$ denote the mild solution to \eqref{SPDE} starting from an initial
	profile $u_0\in\U0$ and an independent noise term $\eta$. Then,
	Condition \eqref{COND:WN} implies that
	$\| u(t)\|_{\U0} \lesssim 1+ \|u_0\|_{\U0},$
	uniformly for all $(u_0\,,t)\in\U0\times\R_+$.
	Conversely, suppose that $u_0\equiv1$,
	$\sigma(z)=z$ for all $z\in\R$,
	and $\sup_{t>0}\|u(t)\|_{\U0}<\infty$. Then,
	Condition \eqref{COND:WN} holds.
\end{lemma}

\begin{remark}
	It follows readily from the forthcoming proof that
	the second, converse, portion of Lemma \ref{lem:bdd:L2} 
	remains valid for example if its technical hypotheses are
	reduced to $\inf_{z\neq0}|\sigma(z)/z|>0$
	and $\inf_{\in\R^d}u_0>0$.
\end{remark}

\begin{proof}
	Recall  that we are assuming  \eqref{COND:WN}.
	Recall also the notation \eqref{p} for the heat kernel $p$.
	We may then observe that
	\begin{equation}\label{mathscr(E)}
		\mathscr{E} = \frac12\int_{\R^d}
		\frac{\mu(\d\xi)}{\|\xi\|^2} = \int_0^\infty(p_{2s}*\Lambda)(0)\,\d s=
		\int_0^\infty \<p_s\,,p_s*\Lambda\>_{L^2(\R^d)}\,\d s.
	\end{equation} In \eqref{lem:bdd:L2},
	the first identity is the definition of $\mathscr{E}$, and the other two
	are consequences of the Fubini-Tonelli theorem and the fact that the heat
	kernel and its Fourier transform both vanish rapidly at infinity. 
	
	Thanks to \eqref{mild} we may write
	$\E(|u(t\,,x)|^2) = Q_1(t\,,x) + Q_2(t\,,x),$
	where
	\begin{align*}
		Q_1(t\,,x) &= \E\left( \left| (p_t*u_0)(x) \right|^2\right),\\
		Q_2(t\,,x) &= \E\left(\left|\int_{(0,t)\times\R^d}
			p_{t-s}(y-x) \sigma(u(s\,,y))\,\eta(\d s\,\d y)\right|^2\right).
	\end{align*}
	By the Cauchy-Schwarz inequality,
	\[
		\adjustlimits \sup_{t>0}\sup_{x\in\R^d}
		\sqrt{Q_1(t\,,x)} \le 
		\adjustlimits \sup_{t>0}\sup_{x\in\R^d}\int_{\R^d} p_t(x-y) 
		\|u_0(y)\|_{L^2(\Omega)}\,\d y\le 
		\|u_0\|_{\U0}.
	\]
	Assume temporarily that $\Lambda$ is a function. Then,
	\begin{equation}
        \begin{split}
		Q_2(t\,,x) \label{Q_2}
		&= \int_0^t\d s\iint_{\R^d\times\R^d}\d y\,\d z\
			p_{t-s}(y-x)p_{t-s}(z-x)  \\
                &\hspace{2.5cm} \times\E\left[\sigma(u(s\,,y)) 
                \sigma(u(s\,,z)) \right]\Lambda(y-z).
        \end{split}
	\end{equation} 
	Because
	$\sigma^2(a)\le ( 1+ \varepsilon^{-1}) |\sigma(0)|^2
	+ (1+\varepsilon) \lip_\sigma^2 a^2$
	for every $\varepsilon>0$ and $a\in\R$,
	the Cauchy-Schwarz inequality implies that, uniformly for all $t,\varepsilon>0$,
	\begin{equation}\label{eq:sigma_bound} 
	\begin{aligned}
		\adjustlimits\sup_{s\in(0,t]}\sup_{y\in\R^d}
		\E\left|\sigma(u(s\,,y)) \sigma(u(s\,,z)) \right|
		&\le ( 1+ \varepsilon^{-1}) |\sigma(0)|^2
		+ (1+\varepsilon) \lip_\sigma^2 M_t^2,
	\end{aligned}
	\end{equation} 
	where $M_t =\sup_{s\in(0,t]}\|u(s)\|_{\U0}$ for all $t>0$. 
	It might help to pause and
	recall that $M_t<\infty$ for every $t>0$ (Proposition \ref{pr:Dalang}).
	Now we may return to the proof and observe that,
	the preceding, \eqref{mathscr(E)}, and \eqref{Q_2} together imply that
	\[
		Q_2(t\,,x)  \le \left[ ( 1+ \varepsilon^{-1}) |\sigma(0)|^2
		+ (1+\varepsilon)\lip_\sigma^2 M_t^2\right]\mathscr{E}.
	\]
	This inequality is valid also in the more general case that $\Lambda$ is a measure, 
	and for essentially the same reasons, except the integral in \eqref{Q_2} has to be written
	long-hand as a convolution integral with respect to the measure $\Lambda$. 
	It follows from the preceding bounds for $Q_1$ and $Q_2$ that
	\[
		\E\left(|u(t\,,x)|^2\right) \le \|u_0\|_{\U0}^2 
		+ \left[ ( 1+ \varepsilon^{-1}) |\sigma(0)|^2
		+ (1+\varepsilon) \lip_\sigma^2 M_t^2 \right]\mathscr{E}.
	\]
	The right-hand side of the above inequality is independent of 
	$x$ and monotonically increasing in $t$.
	Therefore, we optimize the left-hand side over $(t\,,x)$ in order to find that
	\[
		M_t^2 \le \|u_0\|_{\U0}^2 + \left[ ( 1+ \varepsilon^{-1}) |\sigma(0)|^2
		+ (1+\varepsilon) \lip_\sigma^2 M_t^2 \right]\mathscr{E},
	\]
	for all $t,\varepsilon>0$. Thanks to \eqref{COND:WN} we may
	choose $\varepsilon$ small enough so that $(1+\varepsilon)\lip_\sigma^2\mathscr{E}<1$.
	It follows that, for that choice of $\varepsilon=\varepsilon(\mu\,,\sigma)>0$,
	\[
		M_t^2\le \frac{\|u_0\|_{\U0}^2 +  ( 1+ \varepsilon^{-1}) |\sigma(0)|^2
		\mathscr{E}}{ 1- (1+\varepsilon) \lip_\sigma^2 \mathscr{E}},
	\]
	uniformly for all $t>0$. This proves that $\sup_{t>0} M_t\lesssim 1+\|u_0\|_{\U0}$
	where the implied constant is independent of $u_0$, 
	and so completes the proof of the first part of Lemma \ref{lem:bdd:L2}.
	
	For the second part, we suppose that $u_0(x)=1$
	for all $x\in\R^d$, and $\sigma(z)=z$, following the hypotheses of this part. 
	We also assume that \eqref{COND:WN} fails; that is, that $\mathscr{E}\ge1$;
	see \eqref{mathscr(E)}.
	Then, just
	as we computed $\E(|u(t\,,x)|^2)$ in terms of
	$Q_1$ and $Q_2$, we proceed to find that, for all $t>0$ and $x\in\R^d$,
	\begin{align*}
		\E[ u(t\,,0)u(t\,,x)] &=1 + \int_0^t\d s\int_{\R^d}\d y
			\int_{\R^d}\d z\
			p_{t-s}(y) p_{t-s}(z-x)\Lambda(y-z)\times\\
		&\hskip1in\times \E[ u(s\,,y) u(s\,,z)],
	\end{align*}
	provided additionally that $\Lambda$ is a function [the
	more general case is similar but only messier to write.]
	It follows readily from this that $\E[u(t\,,x)u(t\,,y)]$
	depends on $(x\,,y)$ only through $y-x$ and hence
	\begin{align*}
		\E[ u(t\,,0)u(t\,,x)] &=1 + \int_0^t\d s\int_{\R^d}\d y
			\int_{\R^d}\d z\
			p_{t-s}(y) p_{t-s}(z-x)\Lambda(y-z)\times\\
		&\hskip1in\times \E[ u(s\,,0) u(s\,,z-y)].
	\end{align*}
	In other words, $f(t\,,x)=
	\E[u(t\,,0)u(t\,,x)]$ solves the autonomous equation
	\[
		f(t\,,x) = 1 + \int_0^t \<p_s\,,p_s*f(t-s)\Lambda\>_{L^2(\R^d)}\,
		\d s.
	\]
	[This is valid, as is stated, even when $\Lambda$ is a measure.]
	A standard
	fixed point argument shows that $f(t\,,x)=\lim_{n\to\infty}f_n(t\,,x)$
	uniformly in $x\in\R^d$ and locally uniformly in $t>0$, where
	$f_0(t\,,x)=1$ and 
	\[
		f_{n+1}(t\,,x) = 1 + \int_0^t \<p_s\,,p_s*f_n(t-s)\Lambda\>_{L^2(\R^d)}\,
		\d s,
      \]
	for all $t>0$, $x\in\R^d$, and $n\in\N$. Apply induction in order to see
	that 
	\[
		f_n(t\,,x)\ge1
		\qquad\text{for all $t>0$, $x\in\R^d$, and $n\in\N$.}
	\]
	Among other things, this proves that $f>0$ everywhere.
	Therefore we may apply Fatou's
	lemma in order to see that $L(x)=\liminf_{t\to\infty}f(t\,,x)$ satisfies
	\[
		L(x) \ge 1 + \int_0^\infty\<p_s\,,p_s*L\Lambda\>_{L^2(\R^d)}\,
		\d s
		\qquad\forall x\in\R^d.
	\]
	Because $\int_0^\infty \<p_s\,,p_s* \Lambda\>_{L^2(\R^d)}\,\d s=
	\mathscr{E}$ -- see \eqref{mathscr(E)} -- and since we are assuming
	that \eqref{COND:WN} fails -- that is, $\mathscr{E}\ge1$,
	the preceding inductively yields
	\[
		\inf_{x\in\R^d}L(x) \ge 1 + \mathscr{E} + \mathscr{E}^2\ge\cdots=\infty.
	\]
	Thus we see from
	the Cauchy-Schwarz inequality that $\sup_{t>0}\|u(t)\|_{\U0}
	\ge \lim_{t\to\infty} f(t\,,0)=\infty.$ This completes the proof of 
	the lemma.
\end{proof}

We now return to the bulk of our discussion.
From here on out, Condition \eqref{COND:WN} is once
again assumed to hold.

\begin{lemma}\label{lem:L2_cont} 
	Let $u$ denote the mild solution to \eqref{SPDE} starting from 
	an initial profile $u_0\in\U0$ and an independent noise term $\eta$. 
	Then,
	\begin{equation*}
		\adjustlimits
		\lim_{y\to x} \sup_{t\geq  \varepsilon}
		\E\left( |u(t\,, x) - u(t\,, y)|^2 \right) =0
		\qquad\forall x\in\R^d,\ \varepsilon>0.
	\end{equation*}
\end{lemma}

\begin{proof} 
	Let us define 
	\[
		q_t( x\,, y\,; z):= p_t( x-z) - p_t(y-z) \quad \forall t>0,\
		x\,, y\,, z \in \R^d.
	\]
	Thanks to \eqref{mild}, we may write for every $t>0$ and
	$x,y\in\R^d$,
	\begin{equation}\label{EQQ}
		\E \left( |u(t\,, x) - u(t\,, y)|^2\right) = Q_1(t\,, x\,, y) + Q_2(t\,, x\,, y),
	\end{equation}
	where 
	\begin{align*}
	    Q_1(t\,, x\,, y)&:= \left(\int_{\R^d} q_t(x\,, y\, ; z) \E[u_0(z)] \, \d z\right)^2 \\
	    Q_2(t\,, x\,, y)&:=  \E\left(\left|\int_{(0,t)\times\R^d}
				q_{t-s}(x\,, y\, ; z) \sigma(u(s\,,z))\,\eta(\d s\,\d z)\right|^2\right).
	\end{align*}
	Let us first consider $Q_1$.  According to Lemma 6.4 of \cite{CJKS},
	\[
		\int_{\R^d}|p_r(v-a)-p_r(a)|\,\d a\lesssim \min(\|v\|/\sqrt{r}\,,1),
	\]
	uniformly for all $r>0$ and $v\in\R^d$. Therefore, the Cauchy-Schwarz inequality
	implies that
	\[
		Q_1(t\,, x\,, y) \le  \|u_0\|^2_{\U0}\, \left(\int_{\R^d} | q_t(x\,, y\, ; z) | \, \d z\right)^2
		\lesssim \|u_0\|^2_{\U0}\,  \frac{\|x-y\|^2}{t},
	\]
	uniformly for all $t>0$ and $x, y \in \R^d$. In particular,
	\begin{equation}\label{Q1}
		\adjustlimits\lim_{y\to x} \sup_{t\geq \varepsilon} Q_1(t\,, x\,, y) = 0
		\quad\forall \varepsilon>0,\ x\in\R^d.
	\end{equation}
	
	Now we consider $Q_2$. As was done in Lemma \ref{lem:bdd:L2}, let us
	assume temporarily that $\Lambda$ is a function.  In that case,
	\begin{align*}
		Q_2(t\,,x\,, y) &= \int_0^t\d s\iint_{\R^d\times\R^d}\d z\,\d z'\
			q_{t-s}(x\,, y\,, ;z)q_{t-s}(x\,, y\, ; z')  \\
		&\hspace{2.5cm} \times\E\left[\sigma(u(s\,,z)) 
			\sigma(u(s\,,z')) \right]\Lambda(z-z').
	\end{align*}
	Since $\sigma$ is Lipschitz continuous,
	$|\sigma(a)|\leq |\sigma(0)|+\lip_\sigma |a|$ for every $a\in\R$.
	Therefore, Lemma \ref{lem:bdd:L2} implies that 
	\[ 
		\adjustlimits
		\sup_{t\geq0} \sup_{z, z' \in \R^d} \E\left|\sigma(u(s\,,z)) 
		\sigma(u(s\,,z'))\right| <\infty. 
	\]
	Consequently,
	\[
		Q_2(t\,, x\,, y) \lesssim \int_0^\infty \d s
		\iint_{\R^d\times\R^d}\d z\,\d z'\ |q_s(x\,, y\,, ; z)\,  q_s(x\,, y\, ; z') | \Lambda(z-z'). 
	\]
	The preceding remains valid when $\Lambda$ is a measure, except we need to interpret
	the right-hand side (in the only feasible way possible) as a convolution with respect to
	the measure $\Lambda$ in that case.
	Clearly,
	\[
		\lim_{y\to x} q_s(x\,, y\, ; z) =0\quad\forall s>0,\ x,z\in\R^d. 
	\]
	In addition,
	\[
		|q_s(x\,, y\, ; z)\,  q_s(x\,, y\, ; z')| \le 
		[p_s(x-z)+p_s(y-z)][ p_s(x-z')+p_s(y-z')],
	\]
	and, for all $x\,, y \in \R^d$, 
	\begin{align*}
		&\int_0^\infty \d s\iint_{\R^d\times\R^d}\d z\,\d z' p_s(x-z)p_s(y-z') \Lambda(z-z') \\
		&\le\int_0^\infty \<p_s\,,p_s*\Lambda\>_{L^2(\R^d)}\,\d s <\infty,
			\quad \text{thanks to Condition \ref{COND:WN}.}
	\end{align*}  
	Therefore, it follows from the preceding and the dominated convergence theorem that
	\begin{equation}\label{Q2}
		\adjustlimits\lim_{y\to x} \sup_{t>0}  Q_2(t\,, x\,, y) =0
		\quad\forall x\in\R^d.
	\end{equation}
	The lemma follows from \eqref{EQQ}, \eqref{Q1}, and \eqref{Q2}.
\end{proof} 

Part (c) of Proposition \ref{pr:Dalang} is a concrete way to say that the
infinite-dimensional process $\{u(t)\}_{t\ge0}$ is Markovian,
without explicitly identifying a nice space in which that Markov process
resides. 
Because of the minimality of the underlying conditions under which Proposition 
\ref{pr:Dalang} holds [mainly \eqref{cond:Dalang}],
the random-field solution to \eqref{SPDE} has very little regularity beyond 
(joint) measurability.
As such, the above Markov process does not live on a nice-enough path space to allow
full use of well-established methods of ergodic theory as for example is developed in 
Da Prato and Zabczyk \cite{DZ}. Nevertheless, it still makes perfect sense to discuss
invariant measures.  Before we do that, let us recall the Banach space $\U0$
of ``good candidate initial profiles'' for \eqref{SPDE}; see Definition \ref{def:U0}.

\begin{definition}\label{def:inv}
	We say that $u_0\in\U0$ is \emph{invariant} if the solution $u$
	to \eqref{SPDE}, started from $u_0$ and an independent copy
	of $\eta$, has the property that $\{u(t)\}_{t\ge0}$
	is a stationary Markov process. In turn, the law of an invariant 
	initial profile $u_0\in\U0$ is said to be an \emph{invariant measure}
	for \eqref{SPDE}.
\end{definition}

Definition \ref{def:inv} is presented in this way in order
to make its intentions clear.
In order for this definition to be fully meaningful, however, we need to address
the following two questions:
\begin{enumerate}[$\mathscr{Q}_1.$]
      \item What is the state space of the Markov process $u=\{u(t)\}_{t\ge0}$
            in Definition \ref{def:inv}?
       \item What is a ``large-enough'' sigma-algebra on which the laws of
             the elements of $\U0$ are defined?
\end{enumerate}

To answer the first question let us consider \eqref{SPDE},
started at some $u_0\in\U0$ and run using an independent copy of
$\eta$. Recall from Proposition \ref{pr:Dalang} that
$u(t)\in\U0$ for every $t\ge0$. It then follows from 
the joint measurability of $u(t)\in\U0$, and from
Fubini's theorem, that
$u(t)\in L^2(m)$ a.s.~for every finite Borel measure $m$ on $\R^d$
and for all $t>0$. Thus, we can -- and will -- view 
$u=\{u(t)\}_{t\ge0}$ as a Markov process on the state space
\begin{equation}\label{def:S_d}
	\bS_d = \bigcap_{m\in \tilde M_1(\R^d)} L^2(m),
\end{equation}
where $\tilde M_1(\R^d)$ denotes any countable collection of probability 
measures on $\R^d$ that includes at least the following (countably many) probability measures: 
\begin{enumerate}[(i)]
	\item $m = \delta_q$ for every $q \in \Q^d$;
	\item $m(\d x) = |B(0,r)|^{-1} \1_{B(0,r)}(x)\,\d x$ for every $r \in \Q_+$;
	\item $m(\d x) = p_t(x-q)\,\d x$ for every $t \in \Q_+$ and $q \in \Q^d$.
\end{enumerate}

We enumerate $\tilde M_1(\R^d) = \{m_n\}_{n \ge 1}$ and metrize $\bS_d$ with the metric
\[
	d_{\bS_d}(f, g) := \sum_{n=1}^\infty 2^{-n} \,
	\frac{\|f - g\|_{L^2(m_n)}}{1 + \|f - g\|_{L^2(m_n)}},
\]
where we identify $f = g$ in $\bS_d$ whenever $d_{\bS_d}(f\,, g) = 0$; that is, 
whenever $f = g$, $m_n$-almost everywhere for every integer $n \geq 1$. 
Once endowed with the metric $d_{\bS_d}$, the space $\bS_d$ 
is seen to be a Polish space. We denote the resulting Borel $\sigma$-algebra by 
$\sS_d := \mathcal{B}(\bS_d)$. 

It is easy to see that, if $u_0, v_0 \in \U0$, 
then $u_0, v_0 \in \bS_d$ almost surely. 
In addition, since  $\E \|u_0 -v_0\|_{L^2(m_n)} \leq \|u_0-v_0\|_{\U0}$, it follows that 
\begin{equation} \label{eq:metric_U0} 
\E \left[  d_{\bS_d}(u_0, v_0)  \right] \leq \|u_0-v_0\|_{\U0}.
\end{equation}

This addresses question $\mathscr{Q}_1$ above. Let us
pause to observe that every jointly measurable random field $U\in\U0$ can be regarded as a
$\bS_d$--valued random variable (by mapping $\omega\mapsto[U(\omega\,,\cdot)]$), and we will
always work with the jointly measurable representative $U(\omega\,,x)$ on the underlying
probability space. Moreover, the $L^2(\Omega)$-continuity 
in the spatial variable (Lemma \ref{lem:L2_cont}) implies that the finite-dimensional
distributions of such random fields (whenever $t>0$) are determined by their values on the 
countable dense set $\Q^d$. We now turn to question $\mathscr{Q}_2$.

Question $\mathscr{Q}_2$ asks to describe
probability laws on ``measurable'' subsets of $\bS_d$ by identifying what
is meant by ``measurable.'' In other words, we seek to find a natural sigma-algebra
$\sS_d$ of subsets of $\bS_d$ on which the law of every $u_0\in\U0$
can be defined. Since $\bS_d$ is a Polish space, that can be done simply by setting
 \begin{equation}\label{def:sS_d}
	\sS_d : = \mathcal{B}(\bS_d),
\end{equation} 
where $\mathcal{B}(\bS_d)$  denotes the Borel sigma field on $\bS_d$.
In fact,  we can  have that 
\[
	\sS_d=\mathcal{B}(\bS_d)= \bigvee_{m_n\in 
	\tilde M_1(\R^d)} \pi_n^{-1} \left( \mathscr{B}(L^2(m_n))\right),
\] 
where $\pi_n: \bS_d \to L^2(m_n)$ denotes the projection map that is defined via 
$\pi_n(u_0)=u_0$.  This completes our discussion of  Definition \ref{def:inv}. 

Let us pause to mention the following simple result. 

\begin{lemma}\label{lem:DC}
	If $\nu_1,\nu_2 \in  M_1(\bS_d)$ and
	$\int_{\bS_d}f\,\d\nu_1=\int_{\bS_d}f\,\d\nu_2$
	for every bounded and Lipschitz continuous function 
	$f:\bS_d \to\R$, then
	$\nu_1=\nu_2$. 
\end{lemma} 

It is a classical fact that the collection of all bounded Lipschitz continuous functions 
forms a measure-determining class on  Polish spaces.
Lemma \ref{lem:DC} is a ready consequence of the easy-to-verify fact
that the metric space $(\bS_d\,,d_{\bS_d})$ is a Polish space.

With Definitions \ref{def:U0} and \ref{def:inv} in hand, we can
present the following basic result. 

\begin{lemma}\label{lem:Liouville}[The Liouville property]
	If $u_0\in\U0$ is an invariant random field,
	then $\textnormal{mean}(\nu):=\E[u_0(x)]$ does not depend on $x\in\R^d$.
\end{lemma}

\begin{proof}
	We aim to prove that the function $\R^d\ni x\mapsto m(x)=\E[u_0(x)]$ is a constant.
	With this in mind, let $u$ evolve according to \eqref{SPDE}, with initial profile $u_0$
	and an independent copy of the noise $\eta$. 
	Since $\sup_{y\in\R^d}|m(y)|\le\sup_{x\in\R^d}\E(|u_0(x)|^2)^{1/2}<\infty$
	and $u_0$ are jointly measurable, Fubini's theorem and \eqref{mild} together imply  that
	\[
		\E [u(t\,,x)] = \int_{\R^d} p_t(x-y)m(y)\,\d y\qquad
		\forall t>0,\ x\in\R^d.
	\]
	Because $u_0$ is invariant, the left-hand side is equal to 
	$m(x)=\E[u_0(x)]$. In other words, the mean function $m$
	satisfies
	\begin{equation}\label{harmonic}
		m=p_t*m\qquad\forall t>0.
	\end{equation}
	It is well known that \eqref{harmonic} implies that $m$
	is a bounded, harmonic function
	on $\R^d$, hence a constant by Liouville's theorem
	from harmonic analysis.\footnote{%
		Here is a quick proof: \eqref{harmonic}
		and the dominated convergence theorem together ensure that the bounded function
		$m$ is infinitely differentiable. Consequently,
		\[
			0 = \frac{\partial}{\partial t}\int_{\R^d} p_t(x-y)m(y)\,\d y
			=\int_{\R^d}\Delta p_t(x-y)m(y)\,\d y = (p_t*\Delta m)(x)
			\quad\forall t>0,\ x\in\R^d.
		\]
		Send $t\to0+$ to deduce the asserted harmonicity property.}
	This completes the proof.
\end{proof}

\section{Duality and invariance}\label{sec:duality}

It has been well known for a long time that,
in rough terms, the ergodic theory of ``weak-noise systems''
boils down to an $L^2$-analysis using a duality argument; see 
Chapter 10 of Liggett \cite{Liggett2005} as well as 
already-mentioned Ref.s \cite{CoxGreven1994,GrevenDenHollander2007,Shiga1992}
for precise pointers to the literature on particle systems,
and \cite{ChenEisenberg2024,ChenOuyangTindelXia2024,TessitoreZabczyk1998} 
for the more recent works on associated SPDEs.
Our work is also based on $L^2$-analysis and duality, though
our duality argument appears to be particularly robust and simple
to describe.

In order to describe the said duality method,
let us choose and fix an epoch $t>0$ and define a Gaussian noise $\cev{\eta}_t$ on $[0\,,t]\times\R^d$
by setting $\cev{\eta}_t$ to be its time reversal from time $t$; that is,
\[
	\cev{\eta}_t(s\,,x)=\eta(t-s\,,x)\qquad\forall s\in[0\,,t],\ x\in\R^d.
\]
One can make a more precise description of $\cev{\eta}_t$ by 
defining its Wiener (hence also It\^o-Walsh) integrals as follows:
\begin{equation}\label{duality}
	\int_{(0,t)\times\R^d}\varphi(s\,,x)\,\cev{\eta}_t(\d s\,\d x)
	= \int_{(0,t)\times\R^d}\varphi(t-s\,,x)\,\eta(\d s\,\d x),
\end{equation}
whenever $\varphi:\R_+\times\R^d$ is non random and satisfies
$\int_0^t  \left( \varphi(s)*\tilde\varphi(s)*\Lambda\right)(0) \, \d s<\infty,$
where $\tilde{\varphi}(r\,,y)=\varphi(r\,,-y)$ for all $r>0$ and $y\in\R^d$.
We can compute the covariance  functional associated with the right-hand side of
\eqref{duality} in order to see that $\cev{\eta}_t$ has the same law as the restriction of
$\eta$ to $[0\,,t]\times\R^d$.\footnote{This is basically saying that if 
	$\{B_s\}_{s\in[0,t]}$ is Brownian motion by time $t$, 
	then so is $\{ B_t-B_{t-s} \}_{s\in[0,t]}$, valid even if $B$ is
	a Brownian motion on a suitable space of distributions.}

In order to carry out the asserted duality argument for our SPDEs,
let $\{\cev{u}_t(s, x)\}_{(s\,, x)\in [0, t]\times\R^d}$ denote the solution to \eqref{SPDE}
up to time $t$, started from an arbitrary initial profile $u_0\in\U0$ in the sense of 
Definition \ref{def:U0},
but where the noise $\eta$ is replaced by $\cev{\eta}_t$. That is,
$\cev{u}_t$ solves, for all $s\in(0\,,t]$ and $x\in\R^d$,
\begin{equation}\label{mild:t}
	\cev{u}_t(s\,,x) =  (p_s*u_0)(x) +\iint_{(0,s)\times\R^d} 
	p_{s-r}(y-x) \sigma(\cev{u}_t(r\,,y))\,\cev{\eta}_t(\d r\,\d y).
\end{equation}
 Proposition \ref{pr:Dalang} ensures that 
the preceding SPDE on $(0\,,t)\times\R^d$ has a unique solution,
and therefore by uniqueness alone we also have
\begin{equation}\label{duals}
	\{\cev{u}_t(s\,,x)\}_{(s, x)\in [0,t]\times\R^d}
	\ \text{and}\
	\{u(s\,,x)\}_{(s, x) \in [0,t]\times\R^d}
	\ \text{have the same law}.
\end{equation}
In this way, we may (and will) study the asymptotic distributional behavior of the spatial
random field $u(t) = \{u(t\,,x)\}_{x\in\R^d}$, as $t\to\infty$,
by instead studying the asymptotic distributional behavior of the spatial random field 
$\cev{u}_t(t) = \{\cev{u}_t(t\,,x)\}_{x\in\R^d}$. This 
simple but effective idea is the basis of our ``duality argument.'' The end result is the following.

\begin{proposition}\label{pr:duality}
	Consider \eqref{SPDE} started from $u_0\in\U0$ and 
	an independent copy of the noise $\eta$. 
	Assume that the weak-noise condition \eqref{COND:WN} holds,
	and that there exists a spatial random field $v_0\in\U0$ such that
	\begin{equation}\label{cond:homogenize}
		\|p_t*u_0 - v_0\|_{\U0} \to0
		\quad\text{as $t\to\infty$}.
	\end{equation}
	Then, there exists a spatial random field $v\in\U0$ such that
	\begin{equation}\label{eq:duality}
		\| \cev{u}_t(t) - v \|_{\U0} \to0
		\quad\text{as $t\to\infty$}.
	\end{equation}
\end{proposition}

\begin{proof}
	In the first, -- and main -- step of the proof,
	we assume that $\Lambda$ is a function,
	and then aim to prove that $\{\cev{u}_t(t\,,x)\}_{t>0}$ 
	is a Cauchy  sequence in $L^2(\Omega)$,
	uniformly in $x$. With this aim in mind, let us
	write, for every $x\in\R^d$ and $0<s<t<T$,
	\[
		\E\left( | \cev{u}_T(T-s\,,x) - \cev{u}_t(t-s\,,x)|^2\right) = Q_0 + Q_1 + Q_2,
	\]
	where
	\begin{align*}
		Q_0 &= Q_0(t\,,T;s\,,x)
			=\E\left(\left| (p_{T-s}*u_0)(x) - (p_{t-s}*u_0)(x) \right|^2\right),\\
		Q_1 &= Q_1(t\,,T;s\,,x)
			= \int_s^t\d r\int_{\R^d}\d y\int_{\R^d}\d y'\
			p_{r-s}(y-x)p_{r-s}(y'-x)\times\\
		&\hskip2in\times \E\left[\sD(r\,,y)\sD(r\,,y')\right]\Lambda(y-y'), 
	\end{align*}
	where $\sD(r\,,y)= \sigma(\cev{u}_T(T-r\,,y)) - \sigma(\cev{u}_t(t-r\,,y))$, and
	\begin{align*}
		Q_2 &= Q_2(t\,,T;s\,,x)\\
		&= \int_t^T\d r\int_{\R^d}\d y\int_{\R^d}\d y'\
			p_{r-s}(y-x)p_{r-s}(y'-x)\sC(r\,,y\,,y',T)\Lambda(y-y'),
	\end{align*}
	where
	\[
		\sC(r\,,y\,,y',T)
		=\E[\sigma(\cev{u}_T(T-r\,,y))\sigma(\cev{u}_T(T-r\,,y'))].
	\]
	Next we estimate $Q_0$, $Q_1$, and $Q_2$ in turn.
	Before  we begin, let us pause to observe that every $Q_j$ is
	nonnegative because it is the second moment of a random variable.
	We will tacitly use this fact in the coming discussion.  Let 
	\[ 
		f(t\,, T; s) = \| \cev{u}_T(T-s) - \cev{u}_t(t-s)\|_{\U0}^2
		\quad\text{and}\quad
		g(t, T; s) = \sup_{x\in\R^d} Q_0(t\,,T;s\,,x).
	\]
	Thanks to Lemma \ref{lem:bdd:L2} and \eqref{duals}, 
	and because of the global Lipschitz nature of $\sigma$,
	\begin{align*}
		|\sC(r\,,y\,,y',T)| \le \sup_{s>0}
		\| \sigma(u(s))\|_{\U0}^2 =\mathbf{C}<\infty,
	\end{align*}
	From the above we obtain
	\begin{align*}
		&f(t\,,T; s)\\
		&\le g(t\,,T; s) +\lip_{\sigma}^2
			\int_s^t \d r \int_{\R^d} \d y \int_{\R^d} \d y'\
			p_{r-s}(y) \Lambda(y-y') p_{r-s}(y')  f(t\,,T; r) \\
		&\quad + \mathbf{C}\int_t^T \d r \int_{\R^d} \d y 
			\int_{\R^d} \d y'\ p_{r-s}(y) \Lambda(y-y') p_{r-s}(y') .
      \end{align*}
      We rewrite the above, equivalently, as 
      \begin{equation} \label{eq:f:g}\begin{split}
		f(t\,,T;s) &\le g(t\,,T;s) +\lip_{\sigma}^2\int_s^t 
      		\left(p_{2(r-s)}*\Lambda\right)(0)\, f(t\,,T; r)\,\d r \\
      	&\qquad + \mathbf{C}\int_t^T \left(p_{2(r-s)}*\Lambda\right)(0)
            	\,\d r .
      \end{split}\end{equation}
      Thus, we may apply \eqref{eq:f:g} recursively in order to obtain
      the following:
      \begin{align*}
      	&f(t\,,T;s) \le g(t\,,T;s) +\\
		&\  + \sum_{n=1}^{\infty}  \lip_{\sigma}^{2n}
			\int_s^t\d s_1 \int_{s_1}^t\d s_2\cdots \int_{s_{n-1}}^t \d s_n 
			\prod_{i=0}^{n-1} \left(p_{2(s_{i+1}-s_i)}*\Lambda\right)(0)
			\, g(t\,,T; s_n) \\
		&\  +\mathbf{C}
			(1+ \mathscr{Q})
			\int_t^T \left(p_{2(r-s)}*\Lambda\right)(0)\,\d r,
	\end{align*}
	where
	\[
		\mathscr{Q} = \mathbf{C}\sum_{n=1}^{\infty}\lip_{\sigma}^{2n}
		\int_s^t\d s_1\int_{s_1}^t\d s_2 \cdots \int_{s_{n-1}}^t \d s_n\
		\prod_{i=0}^{n-1} \left(p_{2(s_{i+1}-s_i)}*\Lambda\right)(0).
      \]
	Since $u_0\in\U0$, it follows readily that
	$|g(t\,,T;s)|\lesssim1$ uniformly in its variables. Consequently,
      \begin{align*}
		&f(t\,,T;s) \\
		&\lesssim 1 +  \int_t^T  \left(p_{2(r-s)}*\Lambda\right)(0)\,\d r 
			+ \sum_{n=1}^{\infty} 
			\left(\lip_{\sigma}^2\int_0^{\infty}  \left(p_{2r}*\Lambda\right)(0)
			\,\d r\right)^n\\
		&= 1 +  \int_t^T  \left(p_{2(r-s)}*\Lambda\right)(0)\,\d r 
			+ \sum_{n=1}^{\infty} 
			\left(\frac{\lip_{\sigma}^2}{2}\int_0^{\infty}  \left(p_r*\Lambda\right)(0)
			\,\d r\right)^n.
      \end{align*}
      The weak-noise condition \eqref{COND:WN} ensures
     	that the preceding sum is convergent. 
      Next we make two quick observations:
      \begin{enumerate}[(A)]
      	\item Owing to \eqref{cond:homogenize},
            	$\lim_{t\to\infty}\sup_{T>t} g(t\,,T;s) = 0$
            	for every $s>0$; and
		\item By \eqref{COND:WN},
			$ \lim_{t\to\infty}\sup_{T>t}
			\int_t^T  \left(p_{2(r-s)}*\Lambda\right)(0)\,\d r = 0$. 
	\end{enumerate}
	Therefore, (A), (B), and the dominated convergence theorem 
	together imply that
	$\lim_{t\to\infty}\sup_{T>t}f(t\,,T;s) = 0$ for every $s>0$.
      
	It follows from the preceding discussion that, when $\Lambda$ is a function, 
	the mapping $t\mapsto \cev{u}_t(t)$ is Cauchy in the Banach space $\U0$. 
	The same fact holds in the more general case that $\Lambda$ 
	is a measure, except that the expressions for $Q_1$ and $Q_2$ need to be
	now written out in terms of somewhat long convolution integrals against
	the measure $\Lambda$. All else holds true as is stated. 
	This completes the proof.
\end{proof}

Thanks to \eqref{duals},
we can interpret Proposition \ref{pr:duality} loosely as follows: 
Consider the solution to our SPDE \eqref{SPDE}, started from an initial profile
in $\U0$ that is independent of $\eta$ and satisfies \eqref{eq:duality}. Then,
the law of $u(t)$ converges weakly 
to the law of a random field $v\in\U0$ as $t\to\infty$. That in
turn suggests that $v$ must be invariant (Definition \ref{def:inv}).
Our next result (Proposition \ref{pr:duality:inv}) will make this discussion precise.
\begin{remark}\label{rem:W2}
      For every two spatial random fields
      $X,Y\in\U0$, define the following infinite-dimensional variation on the usual
      $W_2$-Wasserstein distance for probability laws (see Villani \cite{Villani}):
      \begin{equation}\label{rho}
      	\rho(X\,,Y) = \inf\|X_1-Y_1\|_{\U0}
      \end{equation}
      where infimum is taken over all couplings $(X_1\,,Y_1)$ 
      whose respective marginal laws are those of $(X\,,Y)$. Let $\nu_X$ and $\nu_Y$ respectively denote the laws of
      $X$ and $Y$, both viewed as probability distributions on 
      $(\bS_d\,,\sS_d)$, and  write $\rho(\nu_X\,,\nu_Y) = \rho(X\,,Y).$
       
	Let us note that  the topology induced by $\rho$
	is finer than that of weak convergence in the sense that
	if $\lim_{n\to\infty}\rho(\nu_n\,,\nu)=0$, then $\nu_n$ converges weakly 
	to $\nu$. To see this, choose and fix a bounded Lipschitz continuous function 
	$f: \bS_d \to \R$. Let $(X_n, X)$ denote a coupling of $(\nu_n, \nu)$ in order to see that
	\[
		\left| \int_{\bS_d} f \,\d\nu_n - \int_{\bS_d} f \,\d\nu \right| 
		= \left| \E f(X_n) - \E f(X) \right| 
		\leq \text{Lip}(f) \, \E [ d_{\bS_d}(X_n, X) ].
	\]
	Take the infimum over all such couplings and recall \eqref{eq:metric_U0} in order
	to see that if $\lim_{n\to\infty}\rho(\nu_n\,,\nu)=0$, then
	\[  
		\left| \int_{\bS_d} f \,\d\nu_n - \int_{\bS_d} f \,\d\nu \right| 
		\leq\lip_f\,  \rho(\nu_n, \nu)\to 0\quad\text{as $n\to\infty$}.
	\]
	In other words, $\nu_n \Rightarrow \nu$, as was desired.
\end{remark}

\begin{proposition}\label{pr:duality:inv}
	In Proposition \ref{pr:duality},
	$\lim_{n\to\infty}\rho(u(t)\,,v)=0$,
	where $\rho$ denotes the Wassstein type distance in \eqref{rho}.
	Moreover, the limiting random field $v$ is invariant (see Def.~\ref{def:inv}).
	Equivalently, the law of $v$, viewed as a probability measure on 
	$(\bS_d\,,\sS_d)$, is an invariant measure for the dynamics
	of \eqref{SPDE}.
\end{proposition}

Proposition \ref{pr:duality:inv} is a modest extension of 
Theorem 1.2 of Chen and Eisenberg \cite{ChenEisenberg2024}
and Theorem 5.10 of Chen et al \cite{ChenOuyangTindelXia2024}
by establishing the existence of invariance measures on the smaller
space $M_1(\bS_d)$ using the weak-noise condition
\eqref{COND:WN} which is optimal 
for $L^2$ methods thanks to Lemma \ref{lem:bdd:L2}.

The standard method for proving the existence of invariant measures is the time-honored
Krylov-Bogoliubov existence theorem \cite[Theorem 3.1.1]{DZ}.
Because the Markov process $\{u(t)\}_{t\ge0}$ lives on the state
space $\bS_d$ that has insufficient topological regularity, in part
since we wish to produce invariant measures on path space, the latter theorem 
does not seem to be applicable in the present 
setting. Thus, we provide an alternative route.

\begin{proof}
	Thanks to \eqref{duals}, we can produce a coupling $(\cev{u}_t(t)\,,v)$
	of two random fields that have the  
	same respective marginals as $(u(t)\,,v)$ for every $t\ge0.$ Therefore,
	Remark \ref{rem:W2} and Proposition \ref{pr:duality} together imply
	the $\lim_{t\to\infty}\rho(u(t)\,,v)=0$. 
	
	We proceed to prove the
	invariance of $v$ next by using a different coupling
	construction.
	Let $u_v$ denote the solution to \eqref{SPDE}
	with initial profile $v$ and using an independent copy of
	the noise $\eta$. Our goal is to prove that
	 \begin{equation}\label{goal:inv}
	 	\E[ f(u_v(s))] = \E [f(v)],
	 \end{equation}
	 for every $s>0$ and all bounded and Lipschitz continuous functions
	 $f:\bS_d \to\R$. This and Lemma \ref{lem:DC} together
	 yield the proposition.
	
	Recall that $u$ denotes the solution to 
	\eqref{SPDE} with the initial data $u_0$,
	and that $u_0$ satisfies \eqref{cond:homogenize}.
	Thus, it follows from Proposition \ref{pr:duality} and Remark \ref{rem:W2} that
	 \begin{equation}\label{E[f(v)]}
		\lim_{t\to\infty}\E[f(u(t+s)] = \E[f(v)],
	\end{equation}
	 for every $s>0$ and all bounded and Lipschitz continuous functions
	 $f:\bS_d \to\R$.
	 Next we compute the left-hand side of \eqref{E[f(v)]} in a different
	 way. 
	 
	 Choose and temporarily fix an arbitrary epoch $t>0$,
	 and let $\vec{u}_t=\{\vec{u}_t(s\,,x)\}_{s>0,x\in\R^d}$ denote the
	 solution to \eqref{SPDE} started from $\cev{u}_t(t)\in\U0$
	 (see Proposition \ref{pr:Dalang}), using an independent
	 copy of the noise $\eta$. It might help to
	 also recall that the space-time random field $\cev{u}_t$
	 is the terminal point in the  SPDE \eqref{mild:t} 
	 with the initial profile $u_0$ that is dual to 
	 the SPDE \eqref{SPDE} that yields $u$. Choose and fix a time $s>0$.
	 In accord with Proposition \ref{pr:duality},
	 \[
	 	\lim_{t\to\infty}\| \vec{u}_t(0)-v\|_{\U0} =
	 	\lim_{t\to\infty} \| \cev{u}_t(t)-v\|_{\U0} =0.
	 \]
	 Therefore, by stability (Lemma \ref{lem:stability}),
	 \[
	 	\adjustlimits\lim_{t\to\infty} \sup_{s>0}\| \vec{u}_t(s)-
		\tilde{u}_v(s) \|_{\U0}=0,
	 \]
	 where $\tilde{u}_v$ solves \eqref{SPDE}, starting from initial profile $v\in\U0$
	 and using the same noise as was used to define $\vec{u}_t$.
	 Since $\tilde{u}_v$ has the same law as $u_v$, it follows from
	 Remark \ref{rem:W2} that
	 \begin{equation}\label{E[f(v)]:2}
	 	\lim_{t\to\infty}\E[ f(\vec{u}_t(s))] = \E[f(u_v(s) )],
	 \end{equation}
	 for every $s>0$ and all bounded and Lipschitz continuous functions $f:\bS_d \to\R$. 
	 Thanks to \eqref{duals} and the Markov property (Proposition
	 \ref{pr:Dalang}(c)), 
	 the conditional law of $u(t+s)$ given the pre-$t$ process
	 $\{u(r)\}_{r\in(0,t)}$ has the same distribution as the conditional law of
	 $\vec{u}_t(s)$ given $\vec{u}_t(0)=\cev{u}_t(t)$. Therefore, 
	 \eqref{duals} yields the identity 
	 $\E[ f(\vec{u}_t(s))]=\E[ f(u(t+s))]$,
	 and \eqref{E[f(v)]:2} yields
	 \begin{equation}\label{E[f(v)]:3}
		\lim_{t\to\infty}\E[ f(u(t+s))] = \E[f(u_v(s))],
	 \end{equation}
	 for every $s>0$ and all bounded and Lipschitz continuous functions $f:\bS_d \to\R$.  
	 Compare the right-hand sides of
	 \eqref{E[f(v)]} and \eqref{E[f(v)]:3} in order to deduce
	 \eqref{goal:inv}, thereby conclude the proof.
\end{proof}

\section{Annealing, and the main result}\label{sec:main}

We have made earlier references to  objects
such as ``annealed invariant measures,'' for example
in the Introduction. Let us spend a few paragraphs and make the quoted terms more
precise. Once that is done, we will finally be ready to state
the main theorem of this paper.

Recall from Section \ref{sec:info} that, among other
things, invariant measures for \eqref{SPDE} are probability measures
on the measure space $(\bS_d\,,\sS_d)$, defined in \eqref{def:S_d}
and \eqref{def:sS_d}, and that the Banach space $(\U0\,,\|\,\cdots\|_{\U0})$
is a good space of ``candidate initial profiles'' for our SPDE (Definition \ref{def:U0}).

\begin{definition}\label{def:anneal}
	Suppose that $u_0\in\U0$ is invariant for the dynamics in \eqref{SPDE}.
	Then we say that the random field $u_0$ -- equivalently its law
	in the sense of Section \ref{sec:info} -- is \emph{annealed} if
	$\sup_{x\in\R^d}\Var[(p_t*u_0)(x)]\to0$ as $t\to\infty$.
\end{definition}

In order to explain Definition \ref{def:anneal},
we will first need to choose an arbitrary number $\theta\in\R$ and
let $V^{(\theta)}$ denote the random-field solution to the SPDE
\eqref{SPDE} started identically at $V^{(\theta)}(0)\equiv\theta$. 
Thanks to \eqref{mild}, this is equivalent to
\begin{equation}\label{v:theta}
	V^{(\theta)}(t\,,x) = \theta + \iint_{(0,t)\times\R^d}
	p_{t-s}(y-x) \sigma(V^{(\theta)}(s\,,y))\,\eta(\d s\,\d y),
\end{equation}
for all $t>0$ and $x\in\R^d$. Because $x\mapsto V^{(\theta)}(0\,,x)$
is identically the constant $\theta$, condition \eqref{cond:homogenize}
holds manifestly whence it follows from Proposition \ref{pr:duality} that
there exists a spatial random field $X_\theta\in\U0$ such that
\begin{equation}\label{X_theta}
	\lim_{t\to\infty} \left\| \cev{V}^{(\theta)}_t(t) - X_\theta
	\right\|_{\U0}=0,
\end{equation}
where, for every fixed $t>0$, the spatial random field
\[
	\cev{V}^{(\theta)}_t=
	\left\{\cev{V}^{(\theta)}_t(s\,,x)\right\}_{s\in(0,t),
	x\in\R^d}
\]
is defined by \eqref{mild:t}
for every $t>0$. To be sure, this means that $\cev{V}^{(\theta)}_t$
solves the following for every $t>0$ fixed:
\[
	\cev{V}^{(\theta)}_t(s\,,x) =  \theta +\iint_{(0,s)\times\R^d} 
	p_{s-r}(y-x) \sigma\left(\cev{V}^{(\theta)}_t(r\,,y)\right)\cev{\eta}_t(\d r\,\d y),
\]
almost surely, simultaneously for all $s\in(0\,,t]$ and $x\in\R^d$.
And of course, $\cev{\eta}_t$ comes from \eqref{duality}.
Thanks to the duality relation \eqref{duals} -- applied with $u$ replaced by
$V^{(\theta)}$ -- and in accord with the definition
\eqref{rho} of the Wasserstein-like distance $\rho$, it follows 
from \eqref{X_theta} that
\begin{equation}\label{V:X_theta}
	\lim_{t\to\infty}\rho\left( V^{(\theta)}(t)\,,X_\theta\right)=0.
\end{equation}

Next, consider an invariant $u_0\in\U0$ that is annealed in the sense
of Definition \ref{def:anneal}, and let $\theta=\E[u_0(x)]$ for one (hence all)
$x\in\R^d$. Let $\{u(t)\}_{t\ge0}$ denote the solution to \eqref{SPDE}
started from $u_0$ using an independent copy of $\eta$, and define $V^{(\theta)}$
as in \eqref{v:theta} using the same noise $\eta$. According to the asymptotic stability
Lemma \ref{lem:u-v} below, and thanks to \eqref{rho},
\[
	\rho\left( u(t) \,, V^{(\theta)}(t)\right)
	\le \|u(t)-V^{(\theta)}\|_{\U0}\to0\quad\text{as $t\to\infty$}.
\]
Proposition \ref{pr:duality:inv} shows that $X_\theta$ is invariant --
equivalently, its law $\nu_\theta$ of $X_\theta$
is an invariant measure (Definition \ref{def:inv}) --
and the preceding two displays together imply that
\[
	\lim_{t\to\infty}\rho\left( u(t)\,,X_\theta\right)=0.
\]
To summarize, if we start \eqref{SPDE}
using an annealed initial profile $u_0\in\U0$ with $\theta=\E[u_0(0)]$,
and an independent copy
of $\eta$, then the solution converges to the invariant measure $\nu_\theta$ regardless
of the details of the law of the initial profile. This suggests (but does not prove) that
the convergence to the invariant measure takes place slowly enough that the effect
of the initial profile are averaged out, since the discussion around $X_\theta$
is a rigorous way to say that only $\theta=\E[u_0(0)]$ is needed
in order to converge to the invariant measure $\nu_\theta$.
The previous sentence is precisely what ``annealed
measures'' are in statistical mechanics, and therefore
justifies the terminology of Definition \ref{def:anneal}.

The probability measures 
$\{\nu_\theta\}_{\theta\in\R}$
are those that were announced in the Informal Theorem of the Introduction,
and play a key role in the ergodic theory of \eqref{SPDE}, as suggested
by the same informal theorem. Thus, let us define them formally
for future reference.

\begin{definition}\label{def:nu_theta}
	For every $\theta\in\R$, let $\nu_\theta$ denote the law of $X_\theta$,
	where $X_\theta$ is defined in \eqref{X_theta}.
\end{definition}

We are now ready to state the main theorem of this paper.
The following states, and extends, formally the Informal Theorem of
the Introduction. Recall that we are tacitly assuming that
the weak-noise condition \eqref{COND:WN} holds.

\begin{theorem}\label{th:main}
	Recall Definition \ref{def:nu_theta}.
	Then, for all $\theta\in\R$:
	\begin{enumerate}[\rm(a)]
		\item $\nu_\theta$ is an annealed and 
			ergodic invariant measure for \eqref{SPDE}.
		\item $\nu_\theta$ is the law of a (strongly) stationary
			random field on $\R^d$.
		\item $\nu_\theta$ is the only mean-$\theta$
			annealed invariant measure for
			\eqref{SPDE}.
		\item If $\theta_1\neq\theta_2$ then $\nu_{\theta_1}$
      			and $\nu_{\theta_2}$ are mutually singular.
	\end{enumerate}
\end{theorem}

We will prove Theorem \ref{th:main} in Section \ref{sec:Proof} below,
after we discuss Conjecture \ref{conj2} in the next section. 

Let us conclude this section with the following characterization of
annealed random fields that are weakly spatially stationary;
this result was referred to in the Introduction.

\begin{proposition}\label{pr:anneal:LLN}
	If $u_0\in\U0$ is an invariant random field that is
	spatially weakly stationary, then it is annealed if and only if 
	it satisfies the law of large numbers \eqref{LLN}. In particular, an invariant
	$u_0\in\U0$ is  annealed if it is both spatially stationary and spatially ergodic. 
\end{proposition}

\begin{proof}
	Let us suppose that $u_0\in\U0$ is a spatially weakly stationary random field,
	and define $\theta = \E [u_0(x)]$ for one, hence all, $x\in\R^d$ thanks
	to weak stationarity. According to the Bochner-Minlos-Schwartz theorem,
	there exists a finite Borel measure $\tau$ on $\R^d$ such that
	\[
		\Cov[u_0(x)\,,u_0(y)] = \int_{\R^d} \e^{i\xi\cdot (x-y)}\,\tau(\d\xi)\qquad\forall
		x,y\in\R^d.
	\]
	In other words, let $\tau$ denote  the spectral measure of $u_0$.
	Also, let $\mathscr{R}$ denote the collection of all
	probability density functions $\psi$ on $\R^d$ 
	that satisfy 
	\[
		\lim_{\|\xi\|\to\infty} \hat{\psi}(\xi) =0.
	\]
	Define
	\[
		\psi_t(x) = t^{-d/2} \psi(x/\sqrt t)\qquad\forall\psi\in\mathscr{R}, \ x\in\R^d,\ t>0.
	\]
	Then, $(\psi*u_0)(x)$ is well-defined thanks to the Fubini-Tonelli theorem (this is
	called a Bochner integral). Note that
	\begin{align*}
		\Var[ (\psi_t*u_0)(x)] &= \E\left(\left| (\psi_t * u_0)(x) - \theta \right|^2\right)\\
		&= \int_{\R^d}\d y\int_{\R^d}\d z\ \psi_t(x-y)\psi_t(x-z)\Cov[u_0(y)\,,u_0(z)  ]\\
		&=\int_{\R^d}\d y\int_{\R^d}\d z\ \psi_t(x-y)\psi_t(x-z)\int_{\R^d} \e^{i\xi\cdot(y-z)}  \,\tau(\d\xi)\\
		&= \int_{\R^d} |\hat{\psi}_t(\xi)|^2\,\tau(\d\xi) = \int_{\R^d} |\hat{\psi}(\xi\sqrt t)|^2\,\tau(\d\xi),
	\end{align*}
	for all $\psi\in\mathscr{R}$ and $t>0$, and regardless of the value of $x\in\R^d$.
	This proves that, for every $\psi\in\mathscr{R}$,
	\begin{equation}\label{psiR}
		\adjustlimits
		\lim_{t\to\infty} \sup_{x\in\R^d} \Var[ (\psi_t*u_0)(x)] 
		= \lim_{t\to\infty} \Var[ (\psi_t*u_0)(0)] = \tau\{0\}.
	\end{equation}
	
	On one hand, we can apply \eqref{psiR} with $\psi=p_1$ (the heat kernel at time $t=1$;
	see \eqref{p}) in order to deduce the following:
	\[
		u_0\text{ is annealed iff }\tau\{0\}=0.
	\]
	On the other hand,  we may apply \eqref{psiR} with
	$\psi$ replaced by $\varphi = 2^{-d}\1_{[-1,1]^d}$ also. This is because
	$\varphi\in\mathscr{R}$ also; indeed,
	\[
		\hat\varphi(\xi) = 2^{-d}\int_{[-1,1]^d}\e^{-i\xi\cdot x}\,\d x =
		\prod_{j=1}^d \frac{\sin|\xi_j|}{|\xi_j|}\quad\forall\xi\in\R^d,
	\]
	where $\sin(0)/0:=1$. Therefore, it follows that
	\[
		u_0\text{ is annealed iff }\adjustlimits
		\lim_{t\to\infty} \sup_{x\in\R^d} \Var[ (\varphi_t*u_0)(x)] = 0.
	\]
	This proves that an invariant $u_0\in\U0$ is annealed iff
	\eqref{LLN} holds since
	\[
		(\varphi_t*u_0)(x) = \frac{1}{(4t)^{d/2}} \int_{[x-\sqrt t,x+\sqrt t]^d} u_0(y)\,\d y,
	\]
	for every $t>0$ and $x\in\R^d$, and because $\E[u_0(y)]$ does not depend on 
	$y\in\R^d$ (Lemma \ref{lem:Liouville}).
	
	In order to complete the proof of the proposition, it remains to show that if
	$u_0\in\U0$ is invariant, as well as spatially stationary and ergodic, then
	$u_0$ is annealed. But this follows immediately from the already-proved
	portion of the proposition, since the mean ergodic theorem ensures that \eqref{LLN}
	holds in this case.
\end{proof}

\section{On Conjecture \ref{conj2}}\label{sec:Linear}

Before we prove Theorem \ref{th:main} we pause to
prove Conjecture \ref{conj2} in two physically interesting cases: Where
$\sigma$ is a non-zero constant; and where $\sigma$ is linear. This is accomplished
respectively in Lemmas \ref{lem:WS:const} and \ref{lem:WS:PAM} below.

In order to study the case that $\sigma$ is a non-zero constant, it suffices to assume
that $\sigma\equiv1$, for one can scale the noise otherwise.

\begin{lemma}\label{lem:WS:const}
	Suppose that $\sigma\equiv1$. 
	Let $u_0 \in \U0$ denote an invariant initial profile that is independent of $\eta$. 
	Then, $u_0$ is weakly spatially  stationary.  
\end{lemma} 

\begin{remark}
	We pause to remark that, in the constant-$\sigma$ case above,
	we do not require $u_0$ to be annealed; compare with Conjecture \ref{conj2}. 
\end{remark}

\begin{proof}
	First, let us assume that
	$\Lambda$ is a function. Since $u_0$ is invariant,
	\[
		\Cov[u(t\,, x)\,, u(t\,, y)] = 
		\Cov[u_0(x)\,, u_0(y)] =:\Gamma(x\,, y)\quad\forall t\ge0,\ x,y\in\R^d.
	\]
	Thanks to the Liouville property (Lemma \ref{lem:Liouville}),
	the assumed invariance of $u_0$
	also implies that $\theta:=\E[u_0(z)]$ 
	does not depend on $z\in\R^d$. Thus,  we may deduce from \eqref{mild} that,
	for all $t>0$ and $x,y\in\R^d$,
	\begin{equation}\label{Gamma}\begin{split}
		\Gamma(x\,, y)
			&=  \int_{\R^d}\d z\int_{\R^d}\d w\ p_t(x-z) p_t(y-w) \Gamma(z\,, w)\\
		&\quad + \int_0^t \d s\int_{\R^d}\d z\int_{\R^d}\d w\  p_{s}(z) p_{s}(w)
			\Lambda(z-w- x+y). 
	\end{split}\end{equation}
	For all $a,x,y\in\R^d$ define
	\[
		\Gamma_a(x\,, y):= \Gamma(x+a\,, y+a),\qquad
		\tilde \Gamma_a(x\,, y):= \Gamma_a(x\,, y) - \Gamma(x\,, y).
	\] 
	A direct computation shows that 
	\begin{equation}\label{tilde_Gamma}
		\widetilde \Gamma_a(x\,, y) = \int_{\R^d}\d w
		\int_{\R^d}\d z\   p_t(x-z) p_t(y-w) \widetilde \Gamma_a(z\,, w),
	\end{equation} 
	for every $a,x,y\in\R^d$.
	
	Recall the definition \eqref{p} of the heat kernel on $\R^d$,
	and define the heat kernel on $\R^{2d}$ as follows:
	\[
		\mathbf{p}_t(x\,,y) : = p_t(x)p_t(y)
		\qquad\forall x,y\in\R^d.
	\]
	We may then write \eqref{tilde_Gamma} as follows:
	\[
		\widetilde{\Gamma}_a(x\,,y) = \int_{\R^{2d}} 
		\mathbf{p}_t((x\,,y) - z)\, \widetilde{\Gamma}_a(z) \, \d z
		\quad\forall a,x,y\in\R^d.
	\]
	Since $\Gamma$ is  bounded uniformly --
	see the proof of Lemma \ref{lem:Liouville} -- it follows
	that $\widetilde \Gamma_a$ is a bounded harmonic function on $\R^{2d}$,
	and hence a constant for every $a\in\R^d$. It remains to prove
	that the said constant is zero. With this aim in mind, let us define  
	\begin{equation}\label{c(a)}
		c(a):=\widetilde \Gamma_a(x\,, y)=\Gamma(x+a\,, y+a) -\Gamma(x\,, y)
		\qquad\forall a,x\,, y \in \R^d.
	 \end{equation}
	 It remains to prove that $c\equiv0$.
	 
	We may apply \eqref{c(a)} twice: Once with $x=y=0$ and
	a second time with $x=y=na$,  in order to see that
	 \[
	 	c((n+1)a) - c(na) = \Gamma(na+a\,,na+a) - \Gamma(na\,,na) = c(a),
	 \]
	 for every $n\in\Z_+$ and $a\in\R^d$. Thus we can see that
	 $c(na)=nc(a)$ for all $n\in\N$ and $a\in\R^d$, and hence
	\[
		\limsup_{n\to \infty} n |c(a)| \le 2\sup_{x\,, y \in \R^d} |\Gamma(x\,, y)|
		\le 2\sup_{x\in\R^d}\Var[u_0(x)]\le\|u_0\|_{\U0}^2<\infty,
	\]  
	thanks to \eqref{c(a)}, the Cauchy-Schwarz inequality, and Definition \ref{def:U0} of the space
	$\U0$. It follows that $c\equiv0$, equivalently that $u_0$ is weakly  spatially stationary,
	as desired. 
	
	The same fact holds in the more general case that $\Lambda$ is a measure,
	but now we interpret \eqref{Gamma}, long-hand, as a somewhat messy convolution integral
	against the measure $\Lambda$. 
\end{proof} 

\begin{lemma}\label{lem:WS:PAM}
	Suppose that $\sigma(z)=z$ for every $z\in\R$. 
	Let $u_0 \in \U0$ denote an invariant initial profile that is annealed and
	independent of $\eta$. 
	Then, $u_0$ is weakly spatially  stationary.  
\end{lemma}

\begin{proof}
	We consider the case that $\Lambda$ is a function; the general
	case follows similarly but after we interpret various integrals that follow
	as convolution integrals against the measure $\Lambda$. Since we have
	explained this several times already, we now assume without loss of much
	generality that $\Lambda$ is a function.
	
	Let $\theta = \E[u_0(x)]$ and recall that $\theta$ does not depend on 
	$x\in\R^d$ [Lemma \ref{lem:Liouville}].
	We now proceed as we did in the proof of Lemma \ref{lem:WS:const}, and
	compute $\Gamma$ as we did in \eqref{Gamma},
	but now find that $\Gamma$ satisfies the following integral equation:
	\begin{equation}\label{eq:pam:cov}\begin{aligned}
		&\Gamma(x\,, y)=H(t\,, x-y) + \iint_{\R^d\times\R^d}\d z'\d z\ p_t(x-z)p_t(y-z') 
			\Gamma(z\,, z')\\
		&\quad + \int_0^t\d s\iint_{\R^d\times\R^d}\d z'\d z\ p_s(x-z) p_s(y-z') 
			\Gamma(z\,, z') \Lambda(z-z'),
	\end{aligned}\end{equation} 
	for all $t>0$ and $x,y\in\R^d$, where 
	\[
		H(t\,, x-y):=\theta^2 \int_0^t \d s\iint_{\R^d\times\R^d}\d z'\d z\  p_s(z)p_s(z') 
		\Lambda(z-z'- (x-y)).
	\]
	Usually, one lets $t\to0$ in \eqref{eq:pam:cov} in order to show that $\Gamma$ solves
	a heat equation with a source term over $\R^{2d}$. Here, we will let $t\to\infty$
	instead.
	
	On one hand, since $u_0$ is annealed [Definition \ref{def:anneal}]
	the second term on the right-hand side of \eqref{eq:pam:cov} tends to 0 as $t \to \infty$.
	Indeed,
	\begin{align*}
		&\left| \iint_{\R^d\times\R^d}
			p_t(x-z)p_t(y-z') \Gamma(z\,, z')\,\d z'\d z\right| 
			=\left|\Cov\left[(p_t * u_0) (x)\,, (p_t*u_0)(y)\right]\right|\\
		&\hskip2.4in\le \sup_{x\in \R^d} \Var\left[ (p_t*u_0) (x) \right] \to 0 \,\, \text{as $t \to \infty$,}
	\end{align*} 
	thanks to the Cauchy-Schwarz inequality.
	On the other hand, the monotone convergence theorem ensures that
	the first term on the right-hand side of \eqref{eq:pam:cov} satisfies
	\begin{align*} 
		\lim_{t\to \infty} H(t\,, x-y) &=   
			\theta^2 \int_0^\infty \d s\iint_{\R^d\times\R^d} \d z'\d z\ 
			p_s(z)p_s(z') \Lambda(z-z'- (x-y))\\
		& \leq  \frac{\theta^2}{2} \int_{\R^d}\frac{\mu(\d z)}{\|z\|^2}   <\infty;
	\end{align*}
	see \eqref{mathscr(E)} for the final inequality.
	Because $\Gamma$ is  bounded uniformly, we are led to the following analogue of 
	\eqref{tilde_Gamma}: 
	\[ 
		\widetilde \Gamma_a (x\,, y) = \int_0^\infty\d s\iint_{\R^d\times\R^d}\d z'\d z\ 
		p_s(x-z) p_s(y-z') \widetilde \Gamma_a(z\,, z') \Lambda(z-z'),
	\]
	for the same $\widetilde\Gamma_a$ as in  \eqref{tilde_Gamma}.  
	Therefore, we can deduce from \eqref{mathscr(E)} that
	\[
		\sup_{x, y \in\R^d}  \left| \widetilde \Gamma_a (x\,, y) \right| \le
		\sup_{x, y \in\R^d}  \left| \widetilde \Gamma_a (x\,, y) \right| 
		\frac{1}{2} \int_{\R^d}\frac{\mu(\d z)}{\|z\|^2}.
	\]
	Because $\lip(\sigma)=1$ in the present
	case, the preceding and weak-noise condition \eqref{COND:WN} together imply that 
	$\widetilde\Gamma_a\equiv0$, which is another way to say that
	$u_0$ is weakly spatially stationary. This completes the proof.
\end{proof}

\section{Proof of Theorem \ref{th:main}}\label{sec:Proof}
Portions of the proof of Theorem \ref{th:main} have been 
completed already via the results and definitions that led us
to this point. We now begin to put the finishing touches on 
that proof. This is done in a few steps.

\subsection{Stability}
In this subsection we state and prove two stability results
(Lemmas \ref{lem:stability} and  \ref{lem:u-v}). Lemma \ref{lem:stability} was used in Proposition \ref{pr:duality:inv} and Lemma \ref{lem:u-v} was mentioned earlier as part of the explanation
of Definition \ref{def:anneal}. 

Let $\{u_{n,0}\}_{n>0}$ denote a sequence of initial profiles in $\U0$
for the dynamics of \eqref{SPDE}. Suppose that $\{u_{n,0}\}_{n>0}$ 
are also defined on the same underlying probability space together with a
noise $\eta$ that is independent of the $\{u_{n,0}\}_{n>0}$. Then Proposition 
\ref{pr:Dalang} ensures that the following system of SPDEs (indexed by a
real number $n>0$)
has a unique solution:
\begin{align*}\left[\begin{aligned}
	&\partial_t u_n = \Delta u_n + \sigma(u_n)\,\eta
		\quad\text{on $(0\,,\infty)\times\R^d$},\\
	&\text{subject to}\quad u_n(0)=u_{n,0}.
\end{aligned}\right.\end{align*}

The following stability result hinges on a standard coupling argument.
It might help to recall the Banach-space norms $\|\,\cdots\|_{\U0}$ from Definition
\ref{def:U0}.

\begin{lemma}\label{lem:stability}
	Suppose that there exists a random field $u_{\infty, 0}$-- independent of the noise $\eta$ as above-- such that
	\begin{equation}\label{cond:stability}
		\|u_{n,0} - u_{\infty, 0}\|_{\U0} \to
		0 \quad\text{as $n\to\infty$}.
	\end{equation}
	Then, $u_{\infty,0}\in\U0$, and 
	$\sup_{t>0}\|u_n(t) - u_\infty(t)\|_{\U0}\to0$
	as $n\to\infty$,
	where $u_\infty$ is the solution to \eqref{SPDE}, using the same
	noise $\eta$ as above, and started from the initial
	profile $u_{\infty, 0}\in\U0$.
\end{lemma}

\begin{proof}
	Condition \eqref{cond:stability} can be recast as
	\[
		\lim_{n,m\to\infty}
		\| u_{n,0} -u_{m,0} \|_{\U0}=0.
	\]
	The remainder of this lemma is a standard stability result,
	except that local-in-time estimates (``$\sup_{t\in(0,T]}$'' where $T>0$ is an arbitrary finite number) 
	replace the more common global-in-time ones (``$\sup_{t>0}$''). Local-in-time
	stability properties of SPDEs are well known, particularly
	when the initial profiles are non random.
	We use the same stability proof
	but perform the required bookkeeping in order
	to show that the particular hypothesis \eqref{COND:WN} implies
	that we in fact have stability globally in time.
	
	To be sure, we can see from \eqref{mild} that, when $\Lambda$ is a function,
	\begin{align*}
		&\E\left( |u_n(t\,,x) - u_m(t\,,x)|^2\right)\le
			\left\| u_{n,0} - u_{m,0}\right\|_{\U0}^2 + \\
		&\quad+\int_0^t\d s\iint_{\R^d\times\R^d}\d y\,\d z\
			p_{t-s}(y-x)p_{t-s}(z-x)\mathscr{B}_{s,n,m}(y\,,z)\Lambda(y-z),
	\end{align*}
	where 
	\begin{align*}
		&\mathscr{B}_{s,n,m}(y\,,z) =\E\left\{ \left( \sigma(u_n(s\,,y)) 
			- \sigma(u_m(s\,,y))\right)\left( \sigma(u_n(s\,,y)) - 
			\sigma(u_m(s\,,y))\right) \right\}\\
		&\le \lip_\sigma^2\,\sup_{y\in \R^d} \E\left( |u_n(s\,,y) - u_m(s\,,y)|^2\right)
			=\lip_\sigma^2\|u_n(s)-u_m(s)\|_{\U0}^2.
	\end{align*} 
	It follows readily from this that
	\[
		\mathscr{E}_{n,m} = \sup_{s>0}\|u_n(s)-u_m(s)\|_{\U0}^2
	\]
	satisfies the following for all $n,m>0$ when $\Lambda$ is a function:
	\begin{align*}
		\mathscr{E}_{n,m} &\le \| u_{n,0} - u_{m,0}\|_{\U0}^2\\
			&\quad + \mathscr{E}_{n,m}\lip_\sigma^2 \int_0^\infty\d s\iint_{\R^d\times\R^d}
			\d y\,\d z\ p_s(y-x)p_s(z-x)\Lambda(y-z)\\
		&= \| u_{n,0} - u_{m,0}\|_{\U0}^2 +
			\frac{\mathscr{E}_{n,m}\lip_\sigma^2}{2} \int_{\R^d}\frac{\mu(\d\xi)}{\|\xi\|^2},
	\end{align*}
	by \eqref{mathscr(E)}.
	The final assertion is true in the more general case that $\Lambda$ is a measure,
	but the multiple integral in the second line above needs to be written as a
	covolution integral against the measure
	$\Lambda$ in that more general setting. In conclusion, 
	we solve for $\mathscr{E}_{n,m}$ in order to deduce
	from \eqref{COND:WN}
	that, for all $n,m>0$,
	\[
		\sup_{s>0}\|u_n(s)-u_m(s)\|_{\U0}^2
		\le\left[1 - \frac{\lip_\sigma^2}{2}
		\int_{\R^d}\frac{\mu(\d\xi)}{\|\xi\|^2} \right]^{-1}
		\| u_{n,0} - u_{m,0}\|_{\U0}^2.
	\]
	Thus, the lemma follows from \eqref{cond:stability} provided that we can prove
	that $\Omega \times\R^d\ni(\omega\,,x)\mapsto u_\infty(t\,,x\,,\omega)$ 
	is measurable for every fixed $t>0$. That is done in a manner 
	similar to the way we proved that
	$u_{\infty, 0}$ has a jointly measurable version. We skip the remaining details.
\end{proof}

Next we conclude this section with another, this time asymptotic, 
stability result. A variation of the following can be found
in Chen et al \cite[Theorem 5.14]{ChenOuyangTindelXia2024}.

\begin{lemma}\label{lem:u-v}
	Let $u$ and $v$ respectively solve \eqref{SPDE} starting from 
	$u_0,v_0\in\U0$, both independent of the same copy of the
	noise $\eta$. Suppose, in addition, that 
	$\|p_t*(u_0-v_0)\|_{\U0}\to0$ as $t\to\infty$.
	Then,  $\| u(t) - v(t)\|_{\U0}\to0$
	as $t\to\infty$.
\end{lemma}

\begin{proof}
	First we consider the case where $\Lambda$ is a function. 
	By \eqref{mild}, for every $t>0$ and $x\in\R^d$:
	\begin{equation}\label{u-v}\begin{split}
		&\E\left( \left| u(t\,,x) - v(t\,,x)\right|^2\right)
			= \E\left( |p_t*(u_0-v_0)(x)|^2\right)+\\
		&\ + \int_0^t\d s\iint_{\R^d\times\R^d}
			\d y\,\d z\ p_{t-s}(y-x)p_{t-s}(z-x)\mathscr{B}_s(y\,,z)
			\Lambda(y-z),
	\end{split}\end{equation}
	where $\mathscr{B}_s(y\,,z) = \E[A_s(y)A_s(z)]$ for
	$A_s(y) = \sigma( u(s\,,y) ) - \sigma( v(s\,,y))$,
	$s>0$, and $y,z\in\R^d$. In order to simplify the typography,
	let us define $\phi(t)= \| u(t) - v(t)\|_{\U0}^2$
	for every $t>0$.
	We may then maximize \eqref{u-v} over $x\in\R^d$ in order
	to find that, for all $t>0$,
	\[
		\phi(t) \le  \|p_t*(u_0-v_0)\|_{\U0}^2
		+\lip_{\sigma}^2\int_0^t \phi(t-s)
		\< p_s\,,p_s*\Lambda\>_{L^2(\R^d)}\,\d s.
	\]
	This preceding discussion remains valid in the more general case
	that $\Lambda$ is a measure, but we have to write 
	the convolution long-hand in the first display of the proof
	in that case. 
	
	Recall that the weak-noise condition
	\eqref{COND:WN} is assumed throughout. If in addition $\lip_\sigma>0$
	then it follows in particular that
	$\int_0^\infty \< p_s\,,p_s*\Lambda\>_{L^2(\R^d)}\,\d s<\infty$;
	see \eqref{mathscr(E)}.  Moreover, thanks to Lemma \ref{lem:bdd:L2},
	$\phi$ is bounded uniformly.
	Therefore, Fatou's lemma implies that
	the finite number $K=\limsup_{t\to\infty}\phi(t)$ satisfies
	\begin{align*}
		K &\le  \lim_{t\to\infty}\|p_t*(u_0-v_0)\|_{\U0}
			+ K\lip_{\sigma}^2\int_0^\infty
			\< p_s\,,p_s*\Lambda\>_{L^2(\R^d)}\,\d s\\
		&= \frac{K\lip_{\sigma}^2}{2}\int_{\R^d}\frac{\mu(\d\xi)}{\|\xi\|^2}.
	\end{align*}
	On one hand, if $\lip_\sigma=0$ then $K=0$ tautologically. On the other
	hand, if $\lip_\sigma>0$, then \eqref{COND:WN} implies that $K=0$.
	In other words, $K$ is always zero, which
	is another way to state the announced result.
\end{proof} 

\subsection{Proof of Theorem \ref{th:main}(a)}
Choose and fix some $\theta\in\R$.
Recall from \eqref{v:theta}
that the solution to \eqref{SPDE}, started identically at $\theta$,
is denoted by $V^{(\theta)}$. Therefore, as was pointed
out earlier,
\eqref{V:X_theta} and Propositions
\ref{pr:duality} and \ref{pr:duality:inv} together imply that the law $\nu_\theta$ of $X_\theta$
is invariant for \eqref{SPDE}. Next we prove that $\nu_\theta$ is annealed.
The proof requires the following.

\begin{lemma}\label{lem:decorr}
	Suppose that $u_0\in\U0$ and $x\mapsto\E[u_0(x)]$ is constant, and let
	$\theta=\E[u_0(x)]$. Then, the solution $u$ to \eqref{SPDE} started
	from initial profile $u_0$, with noise $\eta$ independent of $u_0$, satisfies
	the following, as $t\to\infty$:
	\[\adjustlimits
		\sup_{s>0}\sup_{x\in\R^d}\E\left( | (p_t*u(s))(x) - \theta|^2\right) \le
		\sup_{x\in\R^d}\E\left( | (p_t*u_0)(x) - \theta|^2\right)  + \mathscr{o}(1).
	\]
\end{lemma}

\begin{proof}[Proof of Lemma \ref{lem:decorr}]
	We find it helpful to introduce the function
	\begin{equation}\label{R}
		\sR_t(y\,,z) := \Cov[ u(t\,,y)\,,u(t\,,z)] - \Cov[ (p_t*u_0)(y)\,,(p_t*u_0)(z)],
	\end{equation}
	for all $t>0$ and $y,z\in\R^d$. This notation will be used throughout the proof.
	
	Let us first consider the case that $\Lambda$ is a function.
	Thanks to \eqref{mild} and the independence of $u_0$ and the noise $\eta$,
	we first condition on $u_0$ and then take expectations in order to
	arrive at the pointwise identity,
	\[
		\sR_t(y\,,z)
		=\int_0^t\d s\iint_{\R^d\times\R^d}
		\d v\,\d w\ p_{t-s}(v-y)p_{t-s}(w-z)\mathscr{B}_s(v\,,w)\Lambda(v-w),
	\]
	where $\mathscr{B}_s(v\,,w) = \E[\sigma(u(s\,,v))\sigma(u(s\,,w))].$ Lemma
	\ref{lem:bdd:L2} ensures that $\mathscr{B}_s(v\,,w)$ is bounded uniformly in
	$s>0$ and $v,w\in\R^d$ (see \eqref{eq:sigma_bound}). Therefore,
	\begin{equation} \label{conv_p_p_L}
        \begin{split}
		|\sR_t(y\,,z)| &\lesssim
			\int_0^t\d s\iint_{\R^d\times\R^d}
			\d v\,\d w\ p_s(v-z)p_s(w-y)\Lambda(v-w)\\
		&=\int_0^t (p_{2s}*\Lambda)(z-y)\,\d s,
        \end{split}
	\end{equation}
	uniformly for all $t>0$ and $y,z\in\R^d$. 	
The final line in \eqref{conv_p_p_L} follows because convolution 
is commutative.

	In the more
	general case that $\Lambda$ is a measure, we write the first line, long hand,
	as a convolution against $\Lambda$ in order to see that
	\begin{equation}\label{R_t}
		|\sR_t(y\,,z)|\lesssim\int_0^\infty (p_{2s}*\Lambda)(z-y)\,\d s,
	\end{equation}
	uniformly for all $t>0$ and $y,z\in\R^d$ regardless of whether or not
	$\Lambda$ is a function.  
	
	We can apply \eqref{mild}, first conditionally on $u_0$ and then
	take expectations, in order to see that
	$\E(|(p_t*u(s))(x)-\theta|^2)=\Var[(p_t*u(s))(x)]$ for every $t>0$, $s\ge0$,
	and $x\in\R^d$. Therefore, we write, for every $t,s>0$ and $x\in\R^d$,
	\begin{align*}
		&\Var[(p_t*u(s))(x)] 
			= \E\left(\left| \int_{\R^d} p_t(x-y)
			\left[u(s\,,y) - \theta\right] \d y \right|^2\right)\\
		&=  \iint_{\R^d\times\R^d}
			p_t(x-y)p_t(x-z)\Cov[u(s\,,y)\,,u(s\,,z)]
			\d y\,\d z\\
		&=\iint_{\R^d\times\R^d}
			p_t(x-y)p_t(x-z) \sR_s(y\,,z) \,\d y\,\d z +\\
		&\quad+\iint_{\R^d\times\R^d}
			p_t(x-y)p_t(x-z)\Cov[(p_t*u_0)(y)\,,(p_t*u_0)(z)]
			\d y\,\d z\\
		&=: I_1 + I_2.
	\end{align*}
	We analyse $I_1$ and $I_2$ separately and in turn. Thanks
	to \eqref{R_t}, 
	\begin{align*}
		I_1&\lesssim\iint_{\R^d\times\R^d}
			p_t(x-y)p_t(x-z) \d y\,\d z\int_0^\infty\d s\ (p_{2s}*\Lambda)(y-z)\\
		&=\int_0^\infty (p_{2(t+s)}*\Lambda)(0)\,\d s = \int_0^\infty \d s \int_{\R^d}
			\e^{-2(t+s)\|\xi\|^2}\,\mu(\d\xi)\\
		&=  \frac12 \int_{\R^d} \e^{-2t\|\xi\|^2}\,\frac{\mu(\d\xi)}{\|\xi\|^2} .
	\end{align*}
	As regards $I_2$, we note first that,
	by the Cauchy-Schwarz inequality, the  covariance term in $I_2$
	can be bounded from above by 
	$\sup_{y\in\R^d}\Var[(p_t*u_0)(y)]$, whence
	$I_2\le\sup_{w\in\R^d}\Var[(p_t*u_0)(w)].$
	This effort yields the inequality,
	\[
		\adjustlimits\sup_{s>0}\sup_{w\in\R^d}\Var[(p_t*u(s))(w)] \lesssim
		\sup_{w\in\R^d}\Var[(p_t*u_0)(w)]+ 
		 \int_{\R^d} \e^{-2t\|\xi\|^2}\,\frac{\mu(\d\xi)}{\|\xi\|^2},
	\]
	which implies Lemma \ref{lem:decorr}, thanks to \eqref{COND:WN}
	and the bounded convergence theorem of integration theory.
\end{proof}

We can now return to the proof of Theorem \ref{th:main}(a) and verify that
$X_\theta$ -- equivalently its law $\nu_\theta$ -- is annealed. Recall once
again the random field $V^{(\theta)}$ from \eqref{v:theta}. According
to Lemma \ref{lem:decorr}, 
\begin{align*}
	&\adjustlimits\sup_{s>0}\sup_{x\in\R^d}
		\E\left( \left| (p_t*\cev{V}_s^{(\theta)}(s))(x) - \theta\right|^2\right)
		=\adjustlimits\sup_{s>0}\sup_{x\in\R^d}
		\E\left( \left| (p_t*V^{(\theta)}(s))(x) - \theta\right|^2\right)\\
	&\le \sup_{x\in\R^d}\E\left( \left| (p_t*V^{(\theta)}(0))(x) - \theta\right|^2\right)
		+\mathscr{o}(1)=\mathscr{o}(1)
		\qquad\text{as $t\to\infty$},
\end{align*}
since $V^{(\theta)}(0\,,x)=\theta$ -- hence
$(p_t*V^{(\theta)}(0))(x)=\theta$ --  for all $s,t>0$ and $x\in\R^d$.
At the same time, a few back-to-back appeals to Minkowski's inequality ensure 
that for 
all $s,t>0$ and $x\in\R^d$,
\begin{align*}
	&\left| \E\left( \left| (p_t*X_\theta)(x) - \theta\right|^2\right)^{1/2}
		-\E\left( \left| (p_t*\cev{V}_s^{(\theta)}(s))(x) - 
		\theta\right|^2\right)^{1/2}\right|\\
	&\le \E\left( \left| \left(p_t*\left[ \cev{V}_s^{(\theta)}(s)-X_\theta\right]\right)(x)
   		\right|^2\right)^{1/2} \le\|\cev{V}_s^{(\theta)}(s)-X_\theta\|_{\U0},
\end{align*}
and the final quantity tends to zero when $s\to\infty$; see \eqref{X_theta}.
The preceding two displays together yield
\[
	\sup_{x\in\R^d}\E\left( \left| (p_t*X_\theta)(x) - \theta\right|^2\right)
	\to0\qquad\text{as $t\to\infty$},
\]
which is another way to say that $X_\theta$ -- equivalently, $\nu_\theta$ --
is annealed, since we have already seen that
$x\mapsto \E [X_\theta(x)] =\theta$ is the mean function of $\nu_\theta$. 
It therefore remains to establish the ergodicity of  $\nu_\theta$.

Consider the SPDE \eqref{v:theta}, started from $X_\theta$ and an independent
copy of noise which we denote by $\eta$. Let the resulting solution denoted
by $u$. By the already-proved invariance of $X_\theta$, the infinite-dimensional
stochastic process $\{u(t)\}_{t\ge0}$ is stationary. 

Therefore,
the ergodic theorem ensures that for every bounded and Lipschitz continuous functions $f:\bS_d \to\R$
\[
	\lim_{t\to\infty} \frac{1}{t} \int_0^t f(u(s))\,\d s
	= \E[ f(X_\theta)\mid\sI],
\]
where $\sI$ denote the invariant $\sigma$-algebra for the dynamics of \eqref{SPDE}
driven by $(X_\theta\,,\eta)$. 
Since bounded and Lipschitz continuous functions $f:\bS_d \to\R$  form a determining class on $\bS_d$,
it remains to prove that the limiting random variable
$\E[f(X_\theta)\mid\sI]$ is a constant a.s.~for every bounded and Lipschitz continuous  $f:\bS_d \to\R$.

Let us fix a bounded and Lipschitz function $f:\bS^d \to \R$. It suffices to prove that
\[
	\frac{1}{t} \int_0^t f(u(s))\,\d s \xrightarrow{L^1(\Omega)} 
	\E f(X_\theta)\qquad\text{as $t\to\infty$}.
\]
With this goal in mind, let us first observe that, because
$\E [u_0(x)]=\E [X_\theta(x)]=\theta$
for every $x\in\R^d$,
\[
	\| p_t*(u(0))-V^{(\theta)}(0)\|_{\U0}^2 = 
	\sup_{x\in\R^d}\Var[(p_t*u_0)(x)]\to0
	\quad\text{as $t\to\infty$}.
\]
Therefore, the asymptotic stability Lemma \ref{lem:u-v} yields 
$\|u(s)-V^{(\theta)}(s)\|_{\U0}\to0$ as $s\to\infty$ and hence
\begin{equation}\label{f(u)-f(V)}
	\left( \frac{1}{t} \int_0^t f(u(s))\,\d s
	- \frac{1}{t} \int_0^t f(V^{(\theta)}(s))\,\d s \right)\xrightarrow{ L^1(\Omega)}
	\quad\text{as $t\to\infty$},
\end{equation}
thanks to the bounded convergence theorem.
Because $\| V^{(\theta)}(s)-X_\theta\|_{\U0}\to0$
as $s\to\infty$, the bounded convergence theorem of integration
theory implies
that $\|f(V^{(\theta)}(s))-f(X_\theta)\|_{\U0}\to0$ as $s\to\infty$
as well. Therefore, \eqref{f(u)-f(V)} and the Borel-Cantelli argument 
together imply that 
there exists an unbounded, non-random sequence $0<t(1)<t(2)<\cdots$ such that
\begin{equation}\label{t(n)}\begin{split}
	\lim_{n\to\infty} \frac{1}{t(n)} \int_0^{t(n)} f(u(s))\,\d s
		&= \lim_{n\to\infty} \frac{1}{t(n)} \int_0^{t(n)} f(V^{(\theta)}(s))\,\d s\\
	&= \E[ f(X_\theta)\mid\sI]\quad\text{a.s.}
\end{split}\end{equation}
We might expect the right-most quantity to be a constant
because we expect the other two limit quantities in \eqref{t(n)} 
to belong to the tail $\sigma$-algebra of $\{\eta(t)\}_{t>0}$ --
see \eqref{v:theta} -- which in turn
coincides with the tail $\sigma$-algebra of an infinite-dimensional Brownian motion.
And the latter must be trivial thanks to a suitable version of the Kolmogorov's 0-1 law. 
This turns out to be a somewhat subtle issue, and appears to be true
thanks in addition to Condition \eqref{COND:WN}. We use coupling in order 
to rigorize a variation of the preceding argument.
	
Let $\zeta$ be an independent copy of $\eta$. 
Then, choose and fix a non-random and arbitrary number $T>0$, and define 
$\bar{\eta}(t)=\zeta(t)$ for all $t\in[0\,,T]$ and
$\bar{\eta}(t)=\eta(t)$ for $t>T$. More precisely,
we define, for every non-random $\varphi\in L^2(\R_+\times\R^d)$,
\[
	\int_{\R_+\times\R^d} \varphi\,\d\bar{\eta}
	:=\int_{(0,T]\times\R^d} \varphi\,\d\zeta
	+ \int_{(T,\infty)\times\R^d} \varphi\,\d\eta,
\]
viewed as an identity for Wiener integrals. This defines the noise
$\bar\eta$ as a Gaussian generalized random function with the same
law as $\eta$. Let $\bar{V}^{(\theta)}$ denote the solution
to the SPDE \eqref{v:theta}, except use $\bar\eta$ as the driving
noise in place of $\eta$. It then follows that $\bar{V}^{(\theta)}$
has the same law as $V^{(\theta)}$, but is independent of 
the $\sigma$-algebra $\sF_T$ that is
generated by all random variables of the form
\[
	B_t(\psi) := \int_{(0,t)\times\R^d} \psi(y)\,\eta(\d s\,\d y)\qquad
	\forall \psi\in L^2(\R^d)\,, \, \forall t\in (0, T].
\]
	
We plan to prove that
\begin{equation}\label{goal:theta}
	\lim_{t\to\infty}
	\left\| \bar{V}^{(\theta)}(t) - V^{(\theta)}(t) \right\|_{\U0}=0.
\end{equation}
This will complete the proof. In order to see why, let us first note that
the expectation in \eqref{goal:theta} is bounded uniformly in 
$(t\,,x)\in(0\,,\infty)\times\R^d$;
see Lemma \ref{lem:bdd:L2}.
It therefore
follows from Lemma \ref{lem:bdd:L2}, \eqref{t(n)}, and \eqref{goal:theta} 
that there exists a non-random
subsequence $\{t(n_k)\}_{k=1}^\infty$ of $\{t(n)\}_{n=1}^\infty$
such that
\[
	\lim_{k\to\infty}\frac{1}{t(n_k)} \int_0^{t(n_k)} f(\bar{V}^{(\theta)}(s))\,\d s
	= \E[ f(X_\theta)\mid\sI]\quad\text{a.s.}
\]
It follows from this that $\E[ f(X_\theta)\mid\sI]$ is independent $\sF_T$.
This is because $\bar{V}^{(\theta)}$, and hence the left-hand side of the above identity,
is independent of $\sF_T$.
Since $T>0$ could be chosen to be as large as we want,
this proves that $\E[ f(X_\theta)\mid\sI]$ is independent of $\vee_{s>0}\sF_s$.
At the same time, \eqref{t(n)} and Proposition \ref{pr:Dalang} together imply
that $\E[ f(X_\theta)\mid\sI]$ is measurable with respect to 
$\vee_{s>0}\sF_s$. It follows that $\E[ f(X_\theta)\mid\sI]$ is independent of
itself and hence is a constant. Thus,
it remains to prove \eqref{goal:theta}, as we mentioned earlier. We will do that in the case that
$\Lambda$ is a function. The general case that $\Lambda$ is a measure
is carried out by making adjustments of the type that we have made multiple
times in earlier proofs up to here.
	
Thanks to \eqref{v:theta}, for all $t> T$ and $x\in\R^d$,
\begin{align}\nonumber
	&\E\left(\left| \bar{V}^{(\theta)}(t\,,x) - V^{(\theta)}(t\,,x) \right|^2\right)
		\\\nonumber
	&=2\int_0^T\d s\iint_{\R^d\times\R^d}\d y\,\d z\
		p_{t-s}(y-x)p_{t-s}(z-x) 
		\Lambda(y-z)\E[ A_s(y)A_s(z)]\\
	&\quad+\int_T^t\d s\iint_{\R^d\times\R^d}\d y\,\d z\
		p_{t-s}(y-x)p_{t-s}(z-x) 
		\Lambda(y-z)\mathscr{B}_s(y\,,z),
		\label{ABC}
\end{align}
where
\[
	\mathscr{B}_s(y\,,z) = \E[ ( A_s(y)-\bar{A}_s(y) )
	\cdot (A_s(z)-\bar{A}_s(z))],\quad
	A_s(y) := \sigma(V^{(\theta)}(s\,,y)),
\]
and $\bar{A}_s(y) := \sigma(\bar{V}^{(\theta)}(s\,,y))$
for every $s\ge T$ and $y\in\R^d$.
Thanks to Lemma \ref{lem:bdd:L2} and the Lipschitz continuity of $\sigma$,
\[
	K:=\adjustlimits\sup_{s>0}\sup_{y\in\R^d}\|A_s(y)\|_2<\infty.
\]
And we may apply the Cauchy-Schwarz inequality in order to see that
\[
	\mathscr{B}_s(y\,,z)\le\lip_\sigma^2\psi(s)\qquad\forall s\ge T,\
	y\in\R^d,
\]
where
\[
	\psi(t) = \sup_{x\in\R^d}
	\E\left(\left| \bar{V}^{(\theta)}(t\,,x) - V^{(\theta)}(t\,,x) \right|^2\right)
	\quad\forall t>0.
\]
We emphasize that: (1) $\psi$ is uniformly bounded
essentially because $A_s(y)$ has a bounded
second moment;
and (2) Our goal is to prove that $\psi$ vanishes at infinity.
	
With the preceding goal in mind, we first deduce from \eqref{ABC} that
\begin{align*}
	\psi(t) &\le 2K^2\int_0^T(p_{2(t-s)}*\Lambda)(0)\,\d s
		+\lip_\sigma^2\int_T^t(p_{2(t-s)}*\Lambda)(0)
		\psi(s)\,\d s\\
	&= \frac{2K^2}{\lip_{\sigma}^2}\int_{t-T}^t g(s)\,\d s  +
		\int_T^t g(t-s) \psi(s)\,\d s\quad\forall t>T,
\end{align*}
where
\[
	g(t) =  \lip_\sigma^2 (p_{2t}*\Lambda)(0)
	= \lip_\sigma^2 \int_{\R^d} \e^{-2t\|\xi\|^2}\,\mu(\d\xi)
	\qquad\forall t>0.
\]
Note that $g$ is non increasing, non negative, and integrable. In fact,
\[
	\delta:= \int_0^\infty g(t)\,\d t = \frac{\lip_\sigma^2}{2}\int_{\R^d}\frac{\mu(\d\xi)}{\|\xi\|^2}<1,
\]
thanks to \eqref{COND:WN}. Therefore,
\[
	\psi(t) \le 2\frac{K^2}{\lip_{\sigma}^2}\int_{t-T}^t g(s)\,\d s 
	+ \delta\sup_{s>T}\psi(s) \quad\forall t>T,
\]	
We first send $t\to\infty$, and then $T\to\infty$, in order to see that
\[
	\limsup_{t\to\infty}\psi(t) \le \delta\limsup_{t\to\infty} \psi(t).
\]
Since $\psi$ is uniformly bounded and non negative
and $\delta<1$, the above implies
that $\psi(t)\to0$ as $t\to\infty$. This has the desired result and completes the proof
of part (a) of Theorem \ref{th:main}.\qed

\subsection{Proof of Theorem \ref{th:main}(b)}
Part (b) is equivalent to the statement that $X_\theta$ is stationary.
It might help to recall that this means that
$\{X_\theta(x+y)\}_{x\in\R^d}$ has the same law as 
$X_\theta$ for every $y\in\R^d$ fixed. It is well known that since
$V^{(\theta)}(0\,,x)=\theta$ for all $x$, $V^{(\theta)}(t)$ is a (spatially)
stationary random field for every $t>0$; see Dalang \cite{Dalang1999} or 
Lemma 7.1 of Chen et al \cite{CKNP}. It follows from first principle that
weak limits (in metric $\rho$; see \eqref{rho}) of (spatially) stationary random
fields are (spatially) stationary. Therefore, the stationarity of $X_\theta$
follows from the spatial stationarity of $V^{(\theta)}(t)$ and \eqref{V:X_theta}. 
\qed

\subsection{Proof of Theorem \ref{th:main}(c)}

Suppose that $u_0\in\U0$ is annealed and invariant.
According to Lemma \ref{lem:Liouville}, 
$\theta=\E[u_0(x)]$ does not depend on $x\in\R^d$.
We aim to prove that the law of $u_0$ is $\nu_\theta$.
Let $u$ and $V^{(\theta)}$ solve \eqref{SPDE} using the same
noise $\eta$, independent of $u_0$, and started respectively at $u_0$
and $\theta$; see \eqref{mild} and \eqref{v:theta}.
Now 
\[
	\sup_{x\in\R^d}\Var[(p_t*u_0)(x)] = \| p_t*(u_0 -V^{(\theta)}(0))\|_{\U0}
	\qquad\forall t>0,
\]
and the annealed property of $u_0$ implies that the left-hand side
vanishes at $t=\infty$. Therefore, so does the right-hand size,
and so thanks to the asymptotic stability Lemma \ref{lem:u-v},
\[
	\rho(u(t)\,,V^{(\theta)}) \le
	\| u(t)-V^{(\theta)}(t)\|_{\U0}\to0\quad\text{as $t\to\infty$}.
\]
This proves that
\begin{equation}\label{rho:u:X}
	\rho(u(t)\,,X_\theta)\to0
	\qquad\text{as $t\to\infty$;}
\end{equation}
see \eqref{V:X_theta}. Because of the assumption that $u_0$
is invariant, the law of $u(t)$ is independent of $t$. This
and \eqref{rho:u:X} together imply 
that $\rho(u(t)\,,X_\theta)=0$ for every $t>0$, which is
the same as saying that the law of $u(t)$ is $\nu_\theta$.
This is the desired uniqueness statement of part (c).\qed

\subsection{Proof of Theorem \ref{th:main}(d)}
Choose and fix two distinct, otherwise arbitrary, real numbers $\theta_1$ and $\theta_2$.
We aim to prove that there exists a set $A\in\sS_d$ such that
\begin{equation}\label{goal:nu}
	\nu_{\theta_1}(A)=1
	\quad\text{and}\quad\nu_{\theta_2}(A)=0.
\end{equation}
Equivalently put,
we will construct a measurable set $A$ such that
\[
	\P\{X_{\theta_1}\in A\}=1-\P\{X_{\theta_2}\in A\}=1.
\]
Let $B(x\,,r)=\{y\in\R^d:\ \|y-x\|<r\}$ denote the usual ball of
radius $r>0$ about $x\in\R^d$. 
We will require the following technical lemma.

\begin{lemma}\label{lem:decorr:2}
	Consider \eqref{SPDE} started from $u_0\in\U0$,
	independent of $\eta$. Then, the function $\sR$ in
	\eqref{R} satisfies
	\[
		\sup_{x\in\R^d}\iint_{y,z\in B(x,r)}\sup_{t>0} 
		\left| \sR_t(y\,,z)\right| \d y\,\d z = \mathscr{o}(r^{2d})
		\qquad\text{as $r\to\infty$}.
	\]
\end{lemma}

\begin{proof}
	Define
	\[
		I_r(a) = |B(0\,,r)|^{-1}\1_{B(0,r)}(a) = r^{-d}I_1(a/r)\qquad\forall r>0,\ a\in\R^d,
	\]
	where $|\,\cdots|$ denotes Lebesgue measure on $\R^d$.
	Thanks to \eqref{R_t} and
	a few back-to-back applications of the Tonelli theorem,  
	\begin{align*}
		&\iint_{y,z\in B(x,r)}\sup_{t>0}
			\left|\sR_t(y\,,z)\right| \d y\,\d z\\
		&\lesssim r^{2d}\iint_{\R^d\times\R^d} I_r(y-x)I_r(z-x)\sup_{t>0}
			\left|\sR_t(y\,,z)\right| \d y\,\d z\\
		&\lesssim r^{2d}\iint_{\R^d\times\R^d}\d y\,\d z\,
			I_r(y-x)I_r(z-x)\int_0^\infty (p_{2s}*\Lambda)(z-y)\,\d s\\
		&= r^{2d}\int_0^\infty (p_{2s}* I_r*I_r*\Lambda)(0)\,\d s 
			= r^{2d}\int_0^\infty\d s\int_{\R^d}\mu(\d\xi)\
			\e^{-2s\|\xi\|^2}|\hat{I}_r(\xi)|^2\\
		&= r^{2d}\int_{\R^d}|\hat{I}_1(r\xi)|^2\,\frac{\mu(\d\xi)}{2\|\xi\|^2},
	\end{align*}
	uniformly for all $r>0$ and $x\in\R^d$. Thanks to the Riemann-Lebesgue lemma,
	$\hat{I}_1(r\xi) = |B(0\,,1)|^{-1}\int_{B(0,1)}\exp(ir\xi\cdot y)\,\d y\to0$
	as $r\to\infty$,
	pointwise and boundedly. Therefore, the dominated convergence
	theorem and condition \eqref{COND:WN} together imply the lemma.
\end{proof}

With Lemma \ref{lem:decorr:2} under way, we can now return 
to our proof of Theorem \ref{th:main}(e). Choose and fix an arbitrary
$\theta\in\R$ and define $\sR^{(\theta)}_t(y\,,z)$ as was done
in \eqref{R} but with $u$ specialized to $V^{(\theta)}$. That is,
\[
		\sR^{(\theta)}_t(y\,,z) = \Cov[ V^{(\theta)}(t\,,y)\,,V^{(\theta)}(t\,,z)] 
		= \sR^{(\theta)}_t(0\,,z-y),
\]
for all $t>0$ and $y,z\in\R^d$. 	The second identity above is a consequence
of spatial stationarity; see the already-proved part (b) of Theorem \ref{th:main}.
Lemma \ref{lem:decorr:2} yields the following: As $r\to\infty$,
\begin{equation}\label{AVG}
	\adjustlimits\sup_{t>0}\sup_{x\in\R^d}\iint_{y,z\in B(x\,,r)}
	\left| \sR^{(\theta)}_t(y\,,z)\right|\d y\,\d z= \mathscr{o}(r^{2d}).
\end{equation}
Because $\E[V^{(\theta)}(t\,,x)]=\theta$ for all $t>0$,
$\theta\in\R$, and $x\in\R^d$ --
see \eqref{v:theta} -- it follows from \eqref{AVG},
\eqref{V:X_theta}, and the definition \eqref{rho} of $\rho$ that
\begin{align}\notag
	&\limsup_{r\to\infty}
		\E\left( \left| |B(0,r)|^{-1}\int_{B(0,r)}  X_{\theta_i}(y)\,\d y
		- \theta_i \right|^2\right)
		\label{AVO}\\
	&\ =\adjustlimits\limsup_{r\to\infty}
		\lim_{t\to\infty}
		\E\left( \left| |B(0,r)|^{-1}\int_{B(0,r)}  V^{(\theta_i)}(t\,,y)\,\d y
		- \theta_i \right|^2\right)\\\notag
	&\ = \adjustlimits\limsup_{r\to\infty}
		\lim_{t\to\infty}
		|B(0\,,r)|^{-2}  \iint_{y,z\in B(0,r)}
		\sR^{(\theta_i)}(y\,,z)\,\d y\,\d z =0, \text{ for } i=1,2.
\end{align}
Therefore, by the Borel-Cantelli lemma
there exists a non-random, unbounded sequence $\{r_k\}_{k=1}^\infty$
of positive rational numbers such that $\nu_{\theta_i}(A_i)=1$ for $i=1,2$
where
\[
	A_i := \left\{ h\in \bS_d :\ 
	\lim_{k\to\infty} |B(0,r_k)|^{-1}\int_{B(0,r_k)} h(y)\,\d y = \theta_i \right\}.
\]
We pause to emphasize that $A_1,A_2\in \sS_d$ are measurable
subsets of $\bS_d$; see \eqref{def:S_d} and \eqref{def:sS_d}.
Because $\theta_1\neq\theta_2$,
the set $A=A_1$ has the desired property \eqref{goal:nu}. This 
proves part (e) of Theorem \ref{th:main}, which also concludes the proof of
that theorem.\qed

\section{H\"older continuity}\label{sec:Holder}

The proof of Theorem \ref{th:main} was complicated greatly by the fact
that the said theorem yields a characterization of all annealed, ergodic,
invariant measures in a way that ensures that samples from those invariant
measures are {\it a priori} random fields. That effort would be simplified
if one were able to improve Theorem \ref{th:main} so that the
measures $\nu_\theta$ were probability measures on the space of continuous
functions on $\R^d$. It is not difficult to see that this cannot be done under
the minimal weak-noise hypothesis \eqref{COND:WN}; see the remarks below.
Samy Tindel (personal communications) has asked us whether one can find
more restrictive conditions under which samples from $\nu_\theta$ are
in fact continuous functions. The goal of this section is to answer this question
in the affirmative. In order to do that, let us first
consider the following strengthening,
due to Sanz-Sol\'e and Sarr\`a \cite{SS1999,SS2000} 
of Dalang's condition \eqref{cond:Dalang}:
\begin{equation}\label{SSS}
	\exists\beta\in(0\,,1):\qquad
	\int_{\R^d} \frac{\mu(\d\xi)}{(1+\|\xi\|^2)^\beta} <\infty.
\end{equation} 

Choose and fix some $\theta\in\R$ and recall 
from \eqref{v:theta} that $V^{(\theta)}$
is the solution to our SPDE \eqref{SPDE} starting
from constant initial profile $\theta$.
Among other things, the results of Sanz-Sol\'e and Sarr\`a ({\it ibid.}),
condition \eqref{SSS} ensures that $(t\,,x)\mapsto V^{(\theta)}(t\,,x)$
is H\"older continuous (up to a modification), which in turn implies
immediately that, for every $t>0$ fixed, $x\mapsto V^{(\theta)}(t\,,x)$ lies in
the space $C^\varepsilon_{\textit{loc}}(\R^d)$ of locally $\varepsilon$-H\"older
continuous functions on $\R^d$ for some $\varepsilon>0$. It has been shown
in Khoshnevisan and Sanz-Sol\'e \cite{KS2023} that the condition
\eqref{SSS} for H\"older regularity is unimproveable. Since we are interested in
proving the same result, but at infinitely large times $t\gg1$, we must therefore
rely on \eqref{SSS}. In order to be able to carry the analysis at $t\approx\infty$, 
we will also need the following  weak-noise condition which is more restrictive
than \eqref{COND:WN}:
\begin{equation}\label{cond:Lk}
	\int_{\R^d}\frac{\mu(\d z)}{\|z\|^2}<\infty	\quad\text{and}\quad\lip_\sigma^2\int_{\R^d}\frac{\mu(\d\xi)}{\|\xi\|^2} <\frac{1-\beta}{4d},
\end{equation}
for the same constant $\beta$ that satisfies \eqref{SSS}.
With the preceding under way, we can present the main result of this section.

\begin{theorem}\label{th:Holder}
	Under \eqref{SSS} and \eqref{cond:Lk},
	there exists $\varepsilon>0$ such 
	that $\nu_\theta$ is supported on $C_{\textit{loc}}^\varepsilon(\R^d)$.
\end{theorem}

Before we begin the proof of Theorem \ref{th:Holder}, let us state an unresolved problem.

\begin{conjecture}
	We believe that Theorem \ref{th:Holder} holds
	when the weak-noise constant $(1-\beta)/(4d)$ 
	is replaced by a larger number of the form $c(\beta)/ d$. 
	Since the number $4$ comes from the asymptotically
	sharp constant in the Burkholder-Davis-Gundy (BDG) inequality,
	our conjecture is equivalent to the statement that the 
	asymptotically optimal constant in the BDG inequality does not yield
	the sharp result in the present setting.
\end{conjecture}

Let us begin the proof of Theorem \ref{th:Holder}.
For the remainder of this section we choose and fix an arbitrary
$\theta\in\R$ and write 
\begin{equation}\label{I}
	I(t\,,x) = \int_{(0,t)\times\R^d} p_{t-s}(y-x)
	\sigma\left( V^{(\theta)}(s\,,y) \right) \eta(\d s\,\d y),
\end{equation}
so that $V^{(\theta)}(t\,,x) =\theta+I(t\,,x)$; see \eqref{v:theta}.

\begin{lemma}\label{lem:Lk:bdd}
	Suppose that there exists $k\ge1$ such that
	\begin{equation}\label{cond:Lk:bdd}
		\int_{\R^d}\frac{\mu(\d\xi)}{\|\xi\|^2} <\infty
		\quad\text{and}\quad
		\lip_\sigma^2\int_{\R^d}\frac{\mu(\d\xi)}{\|\xi\|^2} < (4k)^{-1}.
	\end{equation}
	Then, $\sup_{t>0}\sup_{x\in\R^d}\E(|V^{(\theta)}(t\,,x)|^k)<\infty$.
\end{lemma}

\begin{proof}
	Throughout, let us write $\|\,\cdots\|_k$ for the $L^k(\Omega)$-norm
	of whatever appears inside, and let
	\[
		M_t = \sup_{r\le t}\sup_{w\in\R^d}\|V^{(\theta)}(r\,,w)\|_k
		\qquad\forall t>0.
	\]
	According to the theory of Dalang \cite{Dalang1999}, $M_t<\infty$ for every $t>0$.
	Now we proceed with the proof.
	
	First we assume additionally that $\Lambda$ is a function.
	A suitable formulation of the BDG inequality \cite{CBMS} yields the following:
	For every $k\ge1$, $t>0$, and $x\in\R^d$,
	\begin{align}\label{I:B}
		&\|I(t\,,x)\|_k^2 \\\nonumber
		&\le 4k\int_0^t\d s\int_{\R^d}\d z\int_{\R^d}\d y\
			p_{t-s}(y-x)p_{t-s}(z-x)|\mathscr{A}(s\,;\,x\,,y)|^2\Lambda(y-z),
	\end{align}
	where, for all $0<s\leq t$, $x,y\in\R^d$ and $\varepsilon>0$,
	 \begin{align*}
		\mathscr{A}(s\,&;\,x\,,y) = \|\sigma(V^{(\theta)}(s\,,x))
			\sigma(V^{(\theta)}(s\,,y))\|_{k/2}\\
		&\le   |\sigma(0)|^2 + |\sigma(0)|\lip_\sigma  \|V^{(\theta)}(s\,,x) \|_{k/2} 
		+|\sigma(0)|\lip_\sigma  \|V^{(\theta)}(s\,,y) \|_{k/2}\\
		&\qquad \qquad +   \lip_\sigma^2\|V^{(\theta)}(s\,,x)V^{(\theta)}(s\,,y)\|_{k/2}\\
		&\le (1+\varepsilon^{-1}) |\sigma(0)|^2 + (1+ \varepsilon)\lip_\sigma^2\, M_{t}^2.
	\end{align*}
	The second line is a consequence of  the fact that
	$|\sigma(a)|\le |\sigma(0)|+ \lip_\sigma|a|$ for all $a\in\R$  because
	$\sigma$ is Lipschitz continuous.
	Thanks to the preceding and \eqref{mild}, we may then write  
	\[
		M_t^2 \le 2\theta^2 + 8k\left[ (1+\varepsilon^{-1}) |\sigma(0)|^2 + (1+ \varepsilon) \lip_\sigma^2\right] M_t^2\int_0^t
		\< p_s\,,p_s*\Lambda\>_{L^2(\R^d)}\,\d s,
	\] 
    for all $t>0$.
	[In the first line, we also used the elementary inequality, $(a+b)^2 \le 2a^2+2b^2$, valid
	for all $a,b\in\R$.]
	This inequality is valid more generally when $\Lambda$ is a measure. However,
	we must instead write the quantity on the 
	second line in \eqref{I:B} as a long
	convolution integral against the measure $\Lambda$.
	In any case, thanks to the preceding and \eqref{mild},
	\begin{align*}
		M_t^2 &\le  2\theta^2 + 8k\left[ (1+\varepsilon^{-1}) |\sigma(0)|^2 + (1+ \varepsilon) \lip_\sigma^2\right]  M_t^2
			\int_0^t \d s\int_{\R^d}\mu(\d\xi)\ \e^{-2s\|\xi\|^2}\\
		&\le 2\theta^2 + 4k\left[ (1+\varepsilon^{-1}) |\sigma(0)|^2 + (1+ \varepsilon) \lip_\sigma^2\right]  M_t^2\int_{\R^d}\frac{\mu(\d\xi)}{\|\xi\|^2}.
	\end{align*}
	Since $\varepsilon>0$ is arbitrary, solve in order to see that $\sup_{t>0}M_t<\infty$ provided that $k\ge1$ satisfies
	\eqref{cond:Lk:bdd}. This proves the lemma.
\end{proof}

\begin{lemma}\label{lem:cont:x:1}
	If \eqref{SSS} holds for some $\beta\in(0\,,1)$, and 
	\eqref{cond:Lk:bdd} is valid for some $k\ge1$, then
	\[
		\sup_{t>0}\E\left( | V^{(\theta)}(t\,,x) - V^{(\theta)}(t\,,y)|^k \right) \lesssim
		\|x-y\|^{k(1-\beta)},
	\]
	uniformly for all $x,y\in\R^d$. 
\end{lemma}

\begin{proof}
	Throughout this proof, we set
	\[
		\alpha = \frac{1-\beta}{2} \in(0\,,\tfrac12).
	\]
	We can now follow Sanz-Sol\'e and Sarr\`a \cite{SS1999,SS2000}
	improve their estimates by a little bit, and
	specialize them to the present setting. Define
	\[
		Y_\alpha(r\,,z) = \int_{(0,r)\times\R^d} 
		\frac{p_{r-s}(z-y)}{(r-s)^\alpha} \sigma\left(V^{(\theta)}(s\,,y)\right)
		\dot{W}(\d r\,\d y)\quad\forall r>0,z\in\R^d.
	\]
	Recall the stochastic integral process $I$ from \eqref{I}. 
	The stochastic Fubini theorem yields the following, which is basically 
	Sanz-Sol\'e and Sarr\`a's elegant reformulation of the factorization methods from 
	semigroup theory:
	\begin{equation}\label{factorization}
		I(t\,,x)
		= \frac{\sin(\pi\alpha)}{\pi}\int_0^t\d r\int_{\R^d}\d z\
		\frac{p_{t-r}(z-x)}{(t-r)^{1-\alpha}} Y_\alpha(r\,,z)
		\quad\forall t>0,x\in\R^d.
	\end{equation}
	Lemma \ref{lem:Lk:bdd} ensures that
	\[
		L_k=\sup_{t>0}\sup_{x\in\R^d}\|\sigma(V^{(\theta)}(t\,,x))\|_k<\infty.
	\]
	Therefore, in the case that $\Lambda$ is a function, 
	a suitable formulation of the BDG inequality \cite{CBMS}
	yields the following: For every $k\ge1$, $r>0$, and $z\in\R^d$,
	\begin{equation}\label{Y_alpha}\begin{split}
		&\|Y_\alpha(r\,,z)\|_k^2 \le 4kL_k^2\int_0^r\d s\int_{\R^d}\d z\int_{\R^d}\d w\
			\frac{p_r(z)p_r(w)}{r^{2\alpha}}\Lambda(w-z)\\
		&=4kL_k^2\int_0^r \<p_r\,,p_r*\Lambda\>_{L^2(\R^d)} \,\frac{\d r}{r^{2\alpha}}
			\le 4kL_k^2\int_{\R^d}\mu(\d\xi)\int_0^\infty\frac{\d r}{r^{2\alpha}}\
			\e^{-2r\|\xi\|^2}\\
		&=2^{1+2\alpha}kL_k^2\Gamma(1-2\alpha)
			\int_{\R^d} \frac{\mu(\d\xi)}{\|\xi\|^{2\beta}} = C_k^2.
	\end{split}\end{equation}
	Because of \eqref{cond:Lk:bdd}, 
	$\int_{\|\xi\|\le1}\|\xi\|^{-2}\,\mu(\d\xi)<\infty$ Therefore, we can use the
	fact that $0<\beta<1$ in order to deduce from \eqref{SSS} ensures 
	that $C_k<\infty$ in \eqref{Y_alpha}.
	Moreover, the same inequality \eqref{Y_alpha} holds when $\Lambda$ is a measure; 
	only, the first line in \eqref{Y_alpha}
	has to be rewritten as a convolution against the measure $\Lambda$
	when $\Lambda$ cannot be identified with a function; all else remains unchanged.
	Armed with the preceding, \eqref{factorization} and Minkowski's inequality together yield
	\begin{align*}
		\|I(t\,,x)-I(t\,,y)\|_k&\le \int_0^t\d r\int_{\R^d}\d z\
			\frac{|p_{t-r}(z-x)-p_{t-r}(z-y)|}{(t-r)^{1-\alpha}} \|Y_\alpha(r\,,z)\|_k\\
		&\le C_k\int_0^t\frac{\d r}{r^{1-\alpha}}\int_{\R^d}\d z\ |p_r(z-x)-p_r(z-y)|.
	\end{align*}
	Lemma 6.4 of \cite{CJKS} tells us that
	$\int_{\R^d}|p_r(v-a)-p_r(a)|\,\d a\lesssim \min(\|v\|/\sqrt{r}\,,1)$
	uniformly for all $r>0$ and $v\in\R^d$. Consequently, there exists $c>0$ such that
	\[
		\|I(t\,,x)-I(t\,,y)\|_k \le
		c C_k \|x-y\|^{2\alpha}\int_0^\infty\left(\frac{1}{\sqrt s}
		\wedge 1\right)\frac{\d s}{s^{1-\alpha}}
		\propto\|x-y\|^{2\alpha},
	\]
	uniformly for all $t>0$ and $x,y\in\R^d$. This concludes the proof of the lemma.
\end{proof}

\begin{lemma}\label{lem:cont:x:2}
	If \eqref{SSS} holds for some $\beta\in(0\,,1)$, and 
	\eqref{cond:Lk:bdd} is valid for some $k>d/(1-\beta)$, then 
	for every fixed positive $\varepsilon < 1-\beta-(d/k)$,
	\[
		\sup_{t>0}
		\E\left( \sup_{\substack{x,y\in B(0,r)\\ x\neq y}}
		\frac{|V^{(\theta)}(t\,,x)-V^{(\theta)}(t\,,y)|^k}{
		\|x-y\|^{\varepsilon k}}\right) <\infty\qquad\forall r>0.
	\]
\end{lemma}

\begin{proof}
	We appeal to a suitable form of the Kolmogorov continuity theorem,
	as described for example in Proposition A.1 of \cite{DKN}. In order to do that,
	let us first choose and fix a number $\gamma>0$ that satisfies the following:
	\[
		\frac{2d}{k} < \gamma < 1-\beta + \frac{d}{k}.
	\]
	Then we apply Proposition A.1 of \cite{DKN} with the following choice of
	parameters (in the language of that proposition):
	$S=B(0\,,r)$, $f(x)=V^{(\theta)}(t\,,x)$, $\mu=$the Lebesgue measure
	on $B(0\,,r)$, $\varrho(x)=\|x\|$ for all $x\in\R^d$, $\Psi(a)=|a|^k$ for all $a\in\R$,
	and $p(a)=|a|^\gamma$ for all $a\in\R$. That proposition yields 
	the random variables 
	\[
		\mathscr{C}_t = \iint\limits_{B(0,r)\times B(0,r)}
		\frac{|V^{(\theta)}(t\,,x)-V^{(\theta)}(t\,,y)|^k}{%
		\|x-y\|^{\gamma k}}\,\d x\,\d y\qquad\forall t>0,
	\]
	that are finite a.s. In fact, Lemma \ref{lem:cont:x:1} implies that
	\[
		\sup_{t>0}\E\mathscr{C}_t \lesssim\iint\limits_{B(0,r)\times B(0,r)}
		\frac{\d x\,\d y}{\|x-y\|^{k(\gamma-1+\beta)}}<\infty.
	\]
	The conclusion of Proposition A.1 of \cite{DKN} then is that (for every $t>0$ fixed),
	$V^{(\theta)}(t)$ has a modification
	that satisfies the following with probability one: Simultaneously for all
	$x,y\in B(0\,,r)$,
	\[
		|V^{(\theta)}(t\,,x) - V^{(\theta)}(t\,,y)| \lesssim\int_0^{2\|x-y\|}
		\frac{\mathscr{C}^{1/k}}{|B(0\,,u/2)|^{2/k}} u^{\gamma-1}\,\d u
		\propto \mathscr{C}^{1/k}\|x-y\|^{\gamma-(2d/k)},
	\]
	where the implies constant is non random and finite, and does not depend on $(t\,,x\,,y)$.
	This implies the lemma with $\varepsilon = \gamma - (2d/k)$.
\end{proof}

\begin{proof}[Proof of Theorem \ref{th:Holder}]
	If \eqref{cond:Lk} holds, then we can choose
	and fix a real number $k>d/(1-\beta)$ that satisfies
	\eqref{cond:Lk:bdd}. In accord with Lemma \ref{lem:cont:x:2},
	\eqref{V:X_theta}, and Fatou's lemma, for every $\varepsilon < 1-\beta-(d/k)$
	and $r>0$, $X_\theta$ satisfies
	\begin{align*}
		&\E\left( \sup_{\substack{x,y\in B(0,r)\\ x\neq y}}
			\frac{|X_\theta(x)-X_\theta(y)|^k}{
			\|x-y\|^{\varepsilon k}}\right) \\
		&\qquad\le\liminf_{t\to\infty}
			\E\left( \sup_{\substack{x,y\in B(0,r)\\ x\neq y}}
			\frac{|V^{(\theta)}(t\,,x)-V^{(\theta)}(t\,,y)|^k}{
			\|x-y\|^{\varepsilon k}}\right) <\infty,
	\end{align*}
	up to a modification.
	This and a suitable version of the Kolmogorov continuity theorem
	together imply that $\P\{X_\theta\in C^\varepsilon_{\textit{loc}}(\R^d)\}=1$,
	and completes our demonstration.
\end{proof}

\section{The constant coefficient case}\label{sec:Gaussian}
In this section we focus on \eqref{SPDE} in the special case
that $\sigma$ is identically a constant $c_0\neq0$; that is,
\[
	\sigma(x)=c_0\qquad\forall x\in\R^d.
\]
Choose and fix an arbitrary $\theta\in\R$ and note that 
the solution $V^{(\theta)}$ to \eqref{SPDE} specializes
to the following:
$V^{(\theta)}(t\,,x) = \theta + I(t\,,x),$
where
\[
	I(t\,,x) = c_0\iint_{(0,t)\times\R^d}
	p_{t-s}(y-x)\,\eta(\d s\,\d y)\qquad\forall t>0,
	\ x\in\R^d.
\]
Compare with \eqref{v:theta}. Since $V^{(\theta)}$ is a
mean-$\theta$ Gaussian random field, so is $X_\theta$.
Moreover, \eqref{V:X_theta} and the elementary properties
of the metric $\rho$ in \eqref{rho} together imply that
for every $x,y\in\R^d$,
\begin{equation}\label{Cov:Gauss}\begin{split}
	&\Cov[X_\theta(x)\,,X_\theta(y)] = 
		\lim_{t\to\infty}\Cov[V^{(\theta)}(t\,,x)\,,V^{(\theta)}(t\,,y)]\\
	&= \lim_{t\to\infty}
		\E\left[I(t\,,x)I(t\,,y)\right] -\theta^2 =
		c_0^2\int_0^\infty (p_{2s}*\Lambda)( x-y)\,\d s - \theta^2\\
	&=c_0^2\iint_{\R_+\times\R^d}\e^{-i(x-y)\cdot z -2s\|z\|^2}\,\d s\,\mu(\d z)
		- \theta^2\\
	&=\frac{c_0^2}{2}\int_{\R^d} \frac{\cos((x-y)\cdot z)}{\|z\|^2} \,\mu(\d z)
		- \theta^2.
\end{split}\end{equation}
This yields most of the following. Recall that 
$\int_{\R^d}\frac{\mu(\d z)}{\|z\|^2}<\infty$; see \eqref{cond:Lk}.

\begin{corollary}\label{cor:Gauss}
	Suppose that $\sigma\equiv c_0\neq0$. Then, for every $\theta\in\R$,
	the finite-dimensional distributions of $\nu_\theta$ describe
	a spatially stationary Gaussian random field $X_\theta$ with mean $\theta$ and
	covariance function given by \eqref{Cov:Gauss}.
	The random field $X_\theta$ is a.s.~locally H\"older continuous,
	equivalently $\nu_\theta(C^\varepsilon_{\textit{loc}}(\R^d))=1$
	for some $\varepsilon>0$,  iff \eqref{SSS} holds.
\end{corollary}

\begin{proof}[Proof of Corollary \ref{cor:Gauss}]
	We have seen already that every $X_\theta$ is a Gaussian process
	with mean function $\equiv\theta$ and covariance
	given by \eqref{Cov:Gauss}.
	
	Condition  \eqref{SSS} implies that $X_\theta$ is H\"older continuous;
	see Theorem \ref{th:Holder}.  It remains to prove that if \eqref{SSS}
	fails to hold then $X_\theta$ is a.s.~not H\"older continuous. Therefore,
	we assume from now on that \eqref{SSS} fails
	to hold for every $\nu\in(0\,,1)$. In light of \eqref{cond:Dalang},
	this means that
	\[
		\int_{\R^d}\frac{\mu(\d\xi)}{\|\xi\|^{2\nu}}=\infty\qquad\forall
		\nu\in(0\,,1).
	\]
	
	Eq.~\eqref{Cov:Gauss} shows that, for all $x\in\R^d$ and $\theta\in\R$,
	\[
		\E\left( |X_\theta(x)-X_\theta(0)|^2\right)
		= c_0^2\int_{\R^d}\frac{1-\cos(\xi\cdot x)}{\|\xi\|^2}\,\mu(\d\xi).
	\]
	Therefore, the spatial stationarity of the random field $X_\theta$ implies that
	whenever $\nu\in(0\,,1)$, 
	\begin{align*}
		&\int_{\R^d} \E\left( |X_\theta(x)-X_\theta(0)|^2\right)\,
			\frac{\d x}{\|x\|^{d+2-2\nu}}\\
		&\qquad=\frac{c_0^2}{2}\int_{\R^d}\frac{\mu(\d\xi)}{\|\xi\|^2}\int_{\R^d}
			\d x\ \frac{1-\cos(\xi\cdot x)}{\|x\|^{d+2-2\nu}} \propto
			\int_{\R^d}\frac{\mu(\d\xi)}{\|\xi\|^{2\nu}}=\infty.
	\end{align*} 
	Because $X_\theta\in\U0$ [see \eqref{X_theta}],
	the preceding expectation is bounded uniformly in $x$. Therefore,
	\[
		\int_{B(0,1)} \E\left( |X_\theta(x)-X_\theta(0)|^2\right)\,
		\frac{\d x}{\|x\|^{d+2-2\nu}} = \infty\quad\forall \nu\in(0\,,1),
	\]
	regardless of how close $\nu$ is to $0$.
	Since $X_\theta$ is spatially stationary, it follows that for all $\nu
\in(0\,,1)$, 	as close to $0$ as we wish,
	\begin{equation}\label{contra}
		\int_{B(0,1)} \sup_{\|x\|/2<\|b-a\|\le\|x\|}\E\left( 
		|X_\theta(b)-X_\theta(a)|^2\right)\frac{\d x}{\|x\|^{d+2-2\nu}}=\infty.
	\end{equation}
	This proves that for every $\alpha\in(1-\nu\,,1)$,

	\begin{equation}\label{no:Kolm}
		\limsup_{\|x\|\to0^+} \frac{1}{\|x\|^{2\alpha}}
		\sup_{\|x\|/2<\|b-a\|\le\|x\|}
		\E\left(  |X_\theta(b)-X_\theta(a)|^2\right)=\infty,
	\end{equation}
	for otherwise $\mathscr{L}_\alpha=\sup_{\|x\|<1}
	\|x\|^{-2\alpha}\sup_{\|x\|/2<\|b-a\|\le\|x\|}\,\E(|\ldots|^2)<\infty$ whence
	\begin{align*}
		&\int_{B(0,1)} \sup_{\|x\|/2<\|b-a\|\le\|x\|}\E\left( 
			|X_\theta(b)-X_\theta(a)|^2\right)\frac{\d x}{\|x\|^{d+2-2\nu}}\\
		&\hskip2in\le \mathscr{L}_\alpha\int_{B(0,1)}\frac{\d x}{\|x\|^{d+2-2\nu-2\alpha}}
			<\infty,
	\end{align*}
	which would then contradict \eqref{contra}.  Because $\nu\in(0\,,1)$
	is arbitrary, we can see that indeed
	\eqref{no:Kolm} is valid for every $\alpha\in(0\,,1)$. We can deduce
	from \eqref{no:Kolm} that
	\[
		\limsup_{\|b-a\|\to0^+}
		\E\left(  \frac{|X_\theta(b)-X_\theta(a)|^2}{\|b-a\|^{2\alpha}}\right)=\infty,
	\]
	Because $X_\theta(b)-X_\theta(a)$ has a Gaussian distribution,
	\[
		\E\left( |X_\theta(b)-X_\theta(a)|^4\right)
		=3\E\left(|X_\theta(b)-X_\theta(a)|^2\right)^2
		\qquad\forall a,b\in\R^d.
	\]
	Therefore, the Paley-Zygmund inequality yields
	\[
		\P\left\{ |X_\theta(b)-X_\theta(a)|^2
		\ge \frac12 \E\left( |X_\theta(b)-X_\theta(a)|^2\right)\right\}
		\ge\frac{1}{12},
	\]
	uniformly for all $a,b\in\R^d$. It follows the preceding that,
	for every $\alpha\in(0\,,1)$, however small,
	\[
		\P\left\{ \limsup_{\|b-a\|\downarrow0}
		\frac{|X_\theta(b)-X_\theta(a)|}{\|b-a\|^\alpha}=\infty\right\}\ge\frac{1}{12},
	\]
	and so $\nu_\theta(C_{\textit{loc}}^\alpha(\R^d))<1$ for
	every $\alpha\in(0\,,1)$.
\end{proof}

\section{The Parabolic Anderson Model}\label{sec:PAM}
Corollary \ref{cor:Gauss} and Theorem \ref{th:main} together identify
the class of all annealed, ergodic, invariant measures $\{\nu_\theta\}_{\theta\in\R}$
of \eqref{SPDE} when $\sigma$ is identically a constant. In particular, it follows that every $\nu_\theta$ is a Gaussian measure when $\sigma$ is a constant. 
Next we present a concrete Wiener chaos description of samples from the $\nu_\theta$s 
in the case that $\sigma$ is linear. For the remainder of this section we assume
that 
\[
	\sigma(x) = c_0x\qquad\forall x\in\R,
\]
where $c_0\neq0$ is a fixed constant. In the language of Carmona and
Molchanov \cite{CM}, the \eqref{SPDE} in the present linear case
is the \emph{parabolic Anderson model} for the noise $\eta$.

\begin{theorem}\label{th:PAM} 
	Fix an arbitrary $\theta\in\R$. Then, $\nu_\theta$ is the weak 
    limit as $t\to\infty$ of the following $L^2(\Omega)$-convergent
	Wiener chaos representation:
    \begin{align*}
		 &\theta + \theta\sum_{n=1}^{\infty} 
            c_0^n\int_{(0,t)\times \R^d}
			\int_{(0,t_1) \times \R^d} \cdots \int_{(0,t_{n-1}) \times \R^d}\\
		&\qquad p_{t,t_1, t_2,\ldots, t_n}(x,\, y_1,\ldots, y_n)\,
			\eta( \d t_n \, \d y_n) \, \eta(\d t_{n-1}\, \d y_{n-1}) \cdots \eta(\d t_1 \, \d y_1),
	\end{align*} 
	where $p_{t,t_1, t_2,\ldots, t_n}(x,\, y_1,\ldots, y_n)
	= \prod_{i=1}^n p_{t_{i-1}-t_{i}}(y_{i-1}-y_i)$ with $t_0=t$ and $y_0=x$. 
\end{theorem}

In light of \eqref{V:X_theta} and the elementary properties of the metric
$\rho$ [see \eqref{rho}], it suffices to prove that 
$\cev V_t^{(\theta)}(t)\to X_\theta$ in $\U0$ as $t\to\infty$. Choose and
fix a $t>0$. According to \eqref{mild:t}, we may write
\begin{equation}\label{eq:pam:constant}\begin{split}
	\cev V^{(\theta)}_t(t\,,x) &= \theta + c_0
	\int_{(0,t)\times \R^d} p_{t-s}(x-y)\, \cev V_t^{(\theta)}(s\,,y)  \,\cev \eta_t(\d s \,\d y).
\end{split}\end{equation}

Thus, we may expand the right-hand side in order to see that
\begin{align*}
	& \cev V_t^{(\theta)}(t\,,x)
		= \theta  +\theta c_0\int_{(0,t)\times \R^d} p_{t-s}(y-x)  \,\cev \eta_t(\d s\, \d y) + \\
	& + c_0^2\int_{(0,t)\times \R^d} \int_{(0,s)\times \R^d}p_{t-s}(y-x)\, 
		p_{s-r}(z-y)\, \cev V_t^{(\theta)}(r,z)  \, \cev \eta_t(\d r\,\d z)\,\cev \eta_t(\d s\, \d y) .
\end{align*}
Continue to expand, using \eqref{eq:pam:constant}, in order to obtain
\begin{align*}
	&\cev V_t^{(\theta)}(t,\,x)
		= \theta  +\theta c_0 \int_{(0,t)\times \R^d} p_{t-s}(y-x)  \,\cev\eta_t(\d s\, \d y)+  \\
	&\quad + \theta c_0^2\int_{(0,t)\times \R^d} \int_{(0,s)\times \R^d}
		p_{t-s}(y-x)\, p_{s-r}(z-y)\,\cev\eta_t(\d r\,\d z)\,\cev\eta_t(\d s\, \d y) +\\
	&\quad +\cdots +\theta c_0^n\int_{(0,t)\times \R^d} 
		\int_{(0,t_1)\times \R^d}\cdots \int_{(0,t_{n-1})\times \R^d}p_{t-t_1}(y_1-x)\\
	 &\hskip.2in\times\left[\prod_{j=n-1}^{1} p_{t_{j}-t_{j+1}}(y_{j+1}-y_j) \,
	 	\cev\eta_t(\d t_{j+1}\,\d y_{j+1}) \right]\, \cev\eta_t( \d t_1\, \d y_1)  + R_n(t\,,x),
\end{align*}
where
\begin{align}\label{eq:rem:n+1}
	R_n(t\,,x)&= c_0^{n+1}\int_{(0,t)\times \R^d} \int_{(0,t_1)\times \R^d}\cdots 
		\int_{(0,t_n)\times \R^d}p_{t-t_1}(y_1-x)\\\notag
        &\quad \times \cev{V}^{(\theta)}_t(t_{n+1},\,y_{n+1})\,p_{t_{n}-t_{n+1}}(y_{n+1}-y_n)\,\cev\eta_t(\d t_{n+1}\,\d y_{n+1}) \\ \notag
	 &\quad\times\left[\prod_{j=n-1}^{1} p_{t_{j}-t_{j+1}}(y_{j+1}-y_j) \,
	 	\cev\eta_t(\d t_{j+1}\,\d y_{j+1}) \right] \, \cev\eta_t( \d t_1\, \d y_1).\notag
\end{align}

The proof of Theorem \ref{th:PAM} essentially boils down to 
proving that the above remainder term $R_n(t\,,x)$ goes to $0$ in 
$L^2(\Omega)$ uniformly in $(t\,,x)\in\R_+\times\R^d$ as $n\to \infty$. We do that next.  Recall
that \eqref{COND:WN} is in place. 

\begin{lemma}\label{lem:E(R_n)}
	$\lim_{n\to\infty}\sup_{t>0} \|R_n(t)\|_{\U0}=0$.
\end{lemma}

\begin{proof}
	We will prove Lemma \ref{lem:E(R_n)} 
	in the case that $\Lambda$ is a function. The more general
	case, wherein $\Lambda$ is a measure, is proved similarly
	but the proof needs to be written out in terms of long convolutions
	against $\Lambda$,
	and is left to the interested reader. In this case,
	\begin{align}\label{eq:rem}
		&\E\left[R_n(t\,,x)^2\right] = c_0^{2(n+1)} \int_{(0,t)\times \R^d\times \R^d}
			\int_{(0,t_1)\times \R^d \times \R^d}\cdots
			\int_{(0,t_{n})\times \R^d\times \R^d} \\\notag
		&\quad  p_{t-t_1}(y_1-x)  p_{t-t_1}(z_1-x) \, 
			\Lambda(y_1-z_1) \, \d y_1\, \d z_1\, \d t_1  \\\notag
		&\quad  \prod_{j=1}^{n-1} \left[p_{t_{j}-t_{j+1}}(y_{j+1}-y_j)
			\,p_{t_{j}-t_{j+1}}(z_{j+1}-z_j) \, \Lambda(y_{j+1}-z_{j+1})\, 
			\d z_{j+1} \,\d y_{j+1}\,\d t_{j+1}\right] \\\notag
		& \quad p_{t_{n}-t_{n+1}}(y_{n+1}-y_n)\,p_{t_{n}-t_{n+1}}(z_{n+1}-z_n)
			\Lambda(y_{n+1}-z_{n+1})\, \d y_{n+1} \, 
			\d z_{n+1}  \, \d t_{n+1} \\\notag
		&\quad    \E\left[\cev V_t^{(\theta)}(t_{n+1},\,y_{n+1})
			\cev V_t^{(\theta)}(t_{n+1},\,z_{n+1})\right].
	\end{align}
	Thanks to the weak-noise condition \eqref{COND:WN} and Lemma \ref{lem:bdd:L2},
	\[
		\mathscr{C} = \adjustlimits
		\sup_{\substack{t>0 \\ 0<s\leq t}}
        \sup_{y,z\in\R^d}\left| \E\left[\cev V_t^{(\theta)}(s\,,y)
		\cev V_t^{(\theta)}(s\,,z)\right] \right|<\infty.
	\]
	Therefore,
	\begin{align}\notag
		& \int_{(0,t_n) \times \R^d\times \R^d}  
			p_{t_{n}-t_{n+1}}(y_{n+1}-y_n)\,p_{t_{n}-t_{n+1}}(z_{n+1}-z_n)
			\Lambda(y_{n+1}-z_{n+1}) \\\notag
            &\quad \times \E\left[\cev V_t^{(\theta)}(t_{n+1},\,y_{n+1})\cev V_t^{(\theta)}(t_{n+1},\,z_{n+1})\right] 
			 \d y_{n+1} \, 
			\d z_{n+1}  \, \d t_{n+1} \\\notag
		 &\quad \le \mathscr{C}\int_0^{t_n}
            \big(p_{2(t_{n}-t_{n+1})}*\Lambda\big)(y_n-z_n)\, \d t_{n+1}.
	\end{align}
	Next, we observe that
	\begin{align}\label{eq:recursive}
		& \int_{(0,t_{n-1}) \times \R^d\times \R^d}
			\int_0^{t_n}  p_{t_{n-1}-t_n}(y_{n}-y_{n-1})\,
			p_{t_{n-1}-t_n}(z_{n}-z_{n-1})\\ \notag
	&\qquad
		\Lambda(y_{n}-z_{n})\,  \left(
		p_{2(t_{n}-t_{n+1})}*\Lambda\right)(y_n-z_n)\, 
		\d t_{n+1}\, \d z_{n} \, \d y_{n}\, \d t_{n} \\\notag
	&= \int_0^{t_{n-1}} \left\{ p_{2(t_{n-1}-t_{n})}* \left[
		\Lambda(\cdot)\int_0^{t_n}
		\left(p_{2(t_n-t_{n+1})}*\Lambda\right)(\cdot)
		\, \d t_{n+1}\right]\right\}(y_{n-1}-z_{n-1}) \,\d t_n \\\notag
	&\le \int_0^{t_{n-1}}\left[\int_0^{t_n}\left(
		p_{2(t_{n+1}-t_n)}*\Lambda\right)(0)\, \d t_{n+1}\right]
		\cdot \left( p_{2(t_n-t_{n-1})}* \Lambda \right) (y_{n-1} -z_{n-1})
		\, \d t_n.
	\end{align}
	The last inequality is justified because, thanks to \eqref{mathscr(E)},
	the continuous, positive-definite function
	\[
		x\mapsto \left( p_{2(t_n-t_{n+1})} *\Lambda \right)(x) = 
		\int_{\R^d} \e^{-2(t_n-t_{n+1})\|\xi\|^2 - i\xi\cdot x}\,\mu(\d\xi) 
	\]
	is maximized at the origin [Bochner's theorem].
	Owing to \eqref{eq:rem:n+1}, we obtain the following bound by
	applying \eqref{eq:recursive} recursively to \eqref{eq:rem}:
	\begin{align*}
		&\E\left[R_n(t\,,x)^2\right]
			\le \mathscr{C} c_0^{2(n+1)} \int_{(0,t)\times \R^d\times\R^d} 
			\int_0^{t_1}\cdots \int_0^{t_n} \left[\prod_{i=1}^n 
			\left(p_{2(t_i-t_{i+1})}*\Lambda\right)(0) \,\d t_i\right]\\
		& \hskip1in p_{t_1}(y_1-x)  \Lambda(y_1-z_1) 
			p_{t_1}(z_1-x) \,  \d y_1\, \d z_1\, \d t_1 \\
		&\le \mathscr{C}c_0^{2(n+1)} \int_0^t \d t_1\int_0^{t_1}\d t_2\cdots
			\int_0^{t_n}\d t_{n+1} \prod_{i=0}^n 
			\left(p_{2(t_i-t_{i+1})}*\Lambda\right)(0) \\
		&\le \frac{\mathscr{C}c_0^{2(n+1)} }{(n+1)!} 
			\left(\int_0^{\infty} (p_{2s}*\Lambda)(0) \, \d s \right)^{n+1},
	\end{align*} 
	where $t_0:=t$. Thanks to \eqref{mathscr(E)},
	this yields
	\[
		\sup_{t>0} \|R_n(t)\|_{\U0}^2 \le \frac{1}{(n+1)!} 
		\left(\frac{\mathscr{C}c_0^{2} }{2}
		\int_{\R^d} \frac{\mu(\d\xi)}{\|\xi\|^2}\, \d \xi\right)^{n+1},
	\]
	which tends to $0$, as $n\to\infty$,
	under the weak-noise condition \eqref{COND:WN}.
\end{proof}

As was alluded to earlier, Lemma \ref{lem:E(R_n)} covers the bulk of 
the proof of Theorem \ref{th:PAM}. Let us now conclude that
discussion.

\begin{proof}[Proof of Theorem \ref{th:PAM}]
    Thanks to \eqref{X_theta} and Lemma \ref{lem:E(R_n)},
    we may define
    \begin{align*}
		Y_\theta &=\theta + \theta\lim_{t\to\infty}\sum_{n=1}^{\infty} 
            c_0^n\int_{(0,t)\times \R^d}
			\int_{(0,t_1) \times \R^d} \cdots \int_{(0,t_{n-1}) \times \R^d}\\
		&\qquad p_{t,t_1, t_2,\ldots, t_n}(x,\, y_1,\ldots, y_n)\,
			\cev{\eta}_t( \d t_n \, \d y_n) \, 
            \cev{\eta}_t(\d t_{n-1}\, \d y_{n-1}) \cdots 
            \cev{\eta}_t(\d t_1 \, \d y_1),
	\end{align*}
    where the sum and the limit converge in $L^2(\Omega).$
	Lemma \ref{lem:E(R_n)}, \eqref{V:X_theta}, and the definition of $\rho$
	-- see \eqref{rho} -- together imply that 
	\[
		\rho(X_\theta\,,Y_\theta)
		=\lim_{t\to\infty}\rho\left(V^{(\theta)}(t)\,,Y_\theta\right)
		\le\lim_{t\to\infty}
		\left\| \cev V^{(\theta)}_t(t) -Y_\theta\right\|_{\U0} =0.
	\]
	Because $\eta$ and $\cev{\eta}_t$ have the
    same law (see Section \ref{sec:duality}), this is another way to state the theorem; see Definition \ref{def:nu_theta}.
\end{proof}

\section{Spatial ergodicity}\label{sec:ergodicity}

Choose and fix an arbitrary number $\theta\in\R$ and consider the invariant measure
$\nu_\theta$ from Theorem \ref{th:main}, alternatively the invariant random field
$X_\theta$ whose law (viewed as an element of $M_1(\bS_d)$) is $\nu_\theta$.
Evidently, the weak-noise condition \eqref{COND:WN} implies that $\mu\{0\}=0$.
Therefore, the theory of Chen et al \cite{CKNP} implies that the spatial random
field $V^{(\theta)}(t) = \{V^{(\theta)}(t\,,x)\}_{x\in\R^d}$ 
is not only stationary but also ergodic for every fixed $t>0$;
see \eqref{V:X_theta}. [This is a nonlinear version of the ergodic condition of
Theorem 9 of  Maruyama \cite[p.~58]{Maruyama}, valid for Gaussian processes.]
Because of \eqref{V:X_theta} and the fact that the
metric $\rho$ from \eqref{rho} by default induces
a notion of weak convergence for probability measures on $(\bS_d\,,\sS_d)$, 
one might expect that the spatial ergodicity of $V^{(\theta)}(t)$ for large values of $t$
might sometimes transfer to its limit $X_\theta$. That turns out to be
sometimes the case, indeed.

\begin{theorem}\label{th:ergodic}
	Suppose in addition to \eqref{COND:WN} that there exists $k\in\N$ such that
	\begin{equation}\label{cond:ergodic}
		\frac{\lip_\sigma^2}{2}
		\int_{\R^d}\frac{\mu(\d\xi)}{\|\xi\|^2}<\frac{1}{2^{(d+2)/2}
		8k}.
	\end{equation} 
	Then, $\nu_\theta$ is spatially ergodic
	for every $\theta\in\R$; equivalently,
	the spatial random field $X_\theta$ is ergodic
	for every $\theta\in\R$.
\end{theorem}

Thanks to Theorem \ref{th:main}, we can 
identify the support of  $\nu_\theta$, provided additionally that \eqref{cond:ergodic} holds.

\begin{corollary}\label{cor:ergodic}
	If \eqref{cond:ergodic} holds,
	then for every $\theta\in\R$, $\nu_\theta$ is supported on
	\[
		\left\{ h\in \bS_d:\,
		\lim_{\substack{r\to\infty:\\r\in\mathbb{Q}_+}}|B(0\,,r)|^{-1}\int_{B(0,r)}h(x)\,\d x=\theta\right\}.
	\]
\end{corollary}

Recall from Theorem \ref{th:main} that $\{\nu_\theta\in\R\}$ are mutually singular.
This can be seen as an immediate consequence of Corollary \ref{cor:ergodic}
provided additionally that the weak-noise condition \eqref{COND:WN}
is strengthened to \eqref{cond:ergodic}. Let us first prove the corollary,
since much of the proof has been worked out earlier within the preceding results.
After that, we complete the proof of Theorem \ref{th:ergodic}.

\begin{proof}[Proof of Corollary \ref{cor:ergodic}]
	Thanks to \eqref{AVO}, the weak noise condition \eqref{COND:WN} alone
	implies that
	\[
		|B(0,r)|^{-1}\int_{B(0,r)}  X_\theta(y)\,\d y
		\xrightarrow{L^2(\Omega)}\theta\quad\text{as $r\to\infty$}.
	\]
	Theorem \ref{th:ergodic} ensures that $X_\theta$ is (spatially) ergodic
	under the more restrictive condition
	\eqref{cond:ergodic}.
	Therefore, the ergodic theorem ensures that
	the preceding limit holds almost surely as well. This proves
	the corollary.
\end{proof}

\begin{proof}[Proof of Theorem \ref{th:ergodic}]
	Choose and fix an arbitrary $\theta\in\R$ and $k\in\N$ that satisfies
	\eqref{cond:ergodic}.
	According to Chen et al \cite[Theorem 6.4]{CKNP}
	$V^{(\theta)}(t\,,x)\in \mathbb{D}^{1,k}$ for every $t>0$,
	$x\in\R^d$, and integers $k\ge1$, where $\mathbb{D}^{1,k}$ denotes
	the usual space of $k$-times Malliavin-differentiable random variables
	whose Malliavin derivative exists in the $L^2(\Omega)$ sense. Moreover,
	it is shown by Chen et al ({\it ibid}.) that
	the Malliavin derivative of $V^{(\theta)}(t\,,x)$ at $(s\,,y)$ satisfies
	\begin{equation}\label{DV}
		\| D_{s,y} V^{(\theta)}(t\,,x)\|_k \le
		\frac{2C_k\e^{\lambda_0(t-s)}}{\sqrt{
		1 - 2^{(d+2)/2}4k\lip_\sigma^2 
		\left( v_{\lambda_0}*\Lambda \right)(0)}} p_{2(t-s)}(x-y),
	\end{equation}
	where:\footnote{
		To be sure, this result is shown in Ref.~\cite{CKNP} for $\theta=1$
		and the Laplacian replaced by $\frac12\Delta$.
		The stated result is proved in exactly the same way, and with
		the same constants as in \cite[Theorem 6.4]{CKNP}, the way
		the latter result is written. But it is worth mentioning that the heat
		kernel comes from \eqref{p} and not the one for $\frac12\Delta$.%
	}
	\begin{enumerate}[\noindent(i)]
		\item $C_k=\sup_{n\ge0}\sup_{(t,x)\in\R_+\times\R^d}
			\E(|\sigma(V^{(\theta)}_n(t\,,x))|^k)$, where $V^{(\theta)}_n$
			denotes the $n$th-stage approximation to $V^{(\theta)}$
			using Picard iteration; see also (5.13) of Chen et al ({\it ibid}.)
			Fix an arbitrary number $T>0$.
			As a consequence of Dalang's theory, the moments of
			$V^{(\theta)}_n(t\,,x)$ are finite, uniformly in $n\in\N$, $t\in[0\,,T]$,
			and $x\in\R^d$. Thus,
			\[
				C_{k,T} = \adjustlimits
				\sup_{n\ge0}\sup_{(t,x)\in [0, T]\times\R^d}
				\E(|\sigma(V^{(\theta)}_n(t\,,x))|^k)<\infty
				\quad\forall k\in\N,\ T>0.
			\]
			The finiteness of $C_{k,T}$ does not in fact require
			the weak-noise condition \eqref{COND:WN}, only
			\eqref{cond:Dalang}. The same method shows that
			$C_k=\sup_{T>0}C_{k,T}<\infty$ when \eqref{COND:WN}
			holds; see Lemma \ref{lem:bdd:L2} for the case $k=2$,
			for example.
		\item $v_\lambda(x) = \int_0^\infty \exp(-\lambda t)p_t(x)\,\d t$
			for all $\lambda>0$ and $x\in\R^d$; and
		\item $\lambda_0>0$ is large enough to ensure that
			\[
				\lip_\sigma^2(v_{\lambda_0}*\Lambda)(0) < \frac{1}{2^{(d+2)/2}
				4k}.
			\]
	\end{enumerate}
	We have also used the fact that the constant $z_k$ of Chen et al 
	({\it ibid.}) is bounded from above by $2\sqrt{k}$; see \cite[(5.6)]{CKNP}.
	
	The inequality \eqref{DV} can be improved upon a little since
	\[
		\hat{v}_\lambda(\xi) = \int_0^\infty\e^{-\lambda t - t\|\xi\|^2}\,\d t
		=\frac{1}{\lambda+\|\xi\|^2}\qquad\forall \xi\in\R^d,
	\]
	whence
	\[
		(v_\lambda*\Lambda)(0) = \int_{\R^d}\frac{\mu(\d\xi)}{\lambda+\|\xi\|^2}
		\qquad\forall\lambda>0.
	\]
	Thanks to condition \eqref{cond:ergodic}, $\lim_{\lambda\downarrow0} (v_\lambda
	* \Lambda)(0)<[\lip_\sigma^2]^{-1}  [2^{(d+2)/2} 4k]^{-1} $. Therefore, we may let $\lambda_0\downarrow0$
	to find that the Malliavin derivative of $V^{(\theta)}(t\,,x)$ at $(s\,,y)$ satisfies
	\begin{equation}\label{supD}
		\sup_{s,t>0, x,y\in\R^d}\| D_{s,y} V^{(\theta)}(t\,,x)\|_k<\infty.
	\end{equation} 
	This improves Theorem 6.4 of Chen et al ({\it ibid}.) under the additional
	present condition \eqref{cond:ergodic}.
	We can now go through the proof of Theorems 1.6
	and 1.7 of Chen et al \cite{CKNP} and adapt it to the present situation
	as follows.
	
	Suppose $g_1,\ldots,g_k:\R\to\R$
	are bounded and Lipschitz continuous, non-random functions. For every $t>0$
	and $N\in\N$ let
	\[
		V_N(t) = \Var\left( N^{-d}\int_{[0,N]^d}\mathscr{G}(t\,,x)\,\d x\right),
	\]
	where
	\[
		\mathscr{G}(t\,,x) = \prod_{j=1}^k g_j\left( V^{(\theta)}(t\,,x + \zeta^j)\right)\,\d x
		\qquad\forall t>0,\ x,\zeta^1,\ldots,\zeta^k\in\R^d.
	\]
	Define $I_N= N^{-d}\1_{[0,N]^d}$ on $\R^d$ in order to deduce from the display
	that follows (8.3) in Chen et al ({\it ibid.}) the following:
	\[
		\sup_{t>0}V_N(t) \lesssim \sum_{j,\ell=1}^k
		\int_0^\infty\left( p_{2s}*I_N*\tilde{I}_N*\Lambda\right)\left(\zeta^j-\zeta^\ell\right)\d s,
	\]
	uniformly for all $N\in\N$. This improves the display after (8.3) of Chen et al ({\it ibid}.)
	by replacing their local-in-time inequality with a global-in-time inequality, valid thanks
	to \eqref{cond:ergodic} and hence \eqref{supD}. As is wont, we are also writing
	$\tilde{I}_N(x) = I_N(-x)$. The expression inside the integral in the above display
	describes a continuous, positive-definite function. Therefore, it is maximized at
	the origin [Bochner's theorem]. In other words,
	\[
		\sup_{t>0}V_N(t) \lesssim
		\int_0^\infty\left( p_{2s}*I_N*\tilde{I}_N*\Lambda\right)(0)\d s,
	\]
	uniformly for all $N\in\N$ and regardless of $\zeta^1,\ldots,\zeta^k\in\R^d$.
	Therefore, it follows from Plancherel's theorem (see \eqref{mathscr(E)})
	that
	\[
		\sup_{t>0}V_N(t) \lesssim
		\int_0^\infty\d s\int_{\R^d}\mu(\d\xi)\
		\e^{-2s\|\xi\|^2}|\hat{I}_N(\xi)|^2
		\propto \int_{\R^d}
		|\hat{I}_N(\xi)|^2\,\frac{\mu(\d\xi)}{\|\xi\|^2}.
	\]
	Since $\hat{I}_N$ is bounded and vanishes pointwise as $N\to\infty$,
	\eqref{COND:WN} and the dominated convergence theorem together
	yield 
	\begin{equation}\label{limV}
		\adjustlimits
		\lim_{N\to\infty}\sup_{t>0}V_N(t)=0.
	\end{equation}
	Recall the 3-parameter random field $\cev{V}^{(\theta)}
	=\{\cev{V}^{(\theta)}_t(s\,,x)\}_{t>0,s\in(0,t),x\in\R^d}$ from
	\eqref{mild:t}. Thanks to \eqref{duals},
	\[
		V_N(t) = \Var\left( N^{-d}\int_{[0,N]^d}\cev{\mathscr{G}}_t(t\,,x)\,\d x\right),
	\]
	where
	\[
		\cev{\mathscr{G}}_t(t\,,x) = \prod_{j=1}^k g_j\left( 
		\cev{V}^{(\theta)}_t(t\,,x + \zeta^j)\right)\,\d x
		\qquad\forall t>0,\ x,\zeta^1,\ldots,\zeta^k\in\R^d.
	\]
	Therefore, \eqref{limV} and \eqref{V:X_theta} together imply that
	\begin{equation}\label{Vars}
		\lim_{N\to\infty}\Var\left( N^{-d}\int_{[0,N]^d}
		\prod_{j=1}^k g_j\left( X_\theta(x + \zeta^j)\right)\,\d x\right)=0.
	\end{equation}
	Since $X_\theta$ is stationary, the ergodic theorem implies that
	\[
		Y=\lim_{N\to\infty} N^{-d}\int_{[0,N]^d}
		\prod_{j=1}^k g_j\left( X_\theta(x + \zeta^j)\right)\quad\text{exists a.s.~
		and in $L^2(\Omega)$}.
	\]
	Therefore, it follows from \eqref{Vars} that $Y=\E Y$,
	equivalently, $Y$ is non random. Because
	this holds for all mentioned choices of $\zeta^j$s and $g_i$s,
	the ergodicity of $X_\theta$ follows from Lemma 7.2 of
	Chen et al \cite{CKNP}.
\end{proof} 

\subsection*{Acknowledgements}
We thank Professor Samy Tindel whose question led to
Theorem \ref{th:Holder}.

\bibliographystyle{amsplain}
\bibliography{SPDE}

\end{document}